\documentclass[12pt]{amsart}

\usepackage{amssymb}
\usepackage{stmaryrd}
\usepackage{amsmath}
\usepackage{amsthm}
\usepackage{enumerate}
\usepackage{color}
 \usepackage[bookmarks=false, colorlinks]{hyperref}
\hypersetup{colorlinks=true,citecolor=blue, linkcolor = black, urlcolor = black}
\usepackage{setspace}
\usepackage{mathpazo}
\usepackage{fullpage}
\usepackage[all]{xy} 
\usepackage{url}
\usepackage{comment} 	
\usepackage[OT2,T1]{fontenc}
\DeclareSymbolFont{cyrletters}{OT2}{wncyr}{m}{n}
\DeclareMathSymbol{\Sha}{\mathalpha}{cyrletters}{"58}

\makeatletter
\def\@tocline#1#2#3#4#5#6#7{\relax
  \ifnum #1>\c@tocdepth 
  \else
    \par \addpenalty\@secpenalty\addvspace{#2}%
    \begingroup \hyphenpenalty\@M
    \@ifempty{#4}{%
      \@tempdima\csname r@tocindent\number#1\endcsname\relax
    }{%
      \@tempdima#4\relax
    }%
    \parindent\z@ \leftskip#3\relax \advance\leftskip\@tempdima\relax
    \rightskip\@pnumwidth plus4em \parfillskip-\@pnumwidth
    #5\leavevmode\hskip-\@tempdima
      \ifcase #1
       \or\or \hskip 1em \or \hskip 2em \else \hskip 3em \fi%
      #6\nobreak\relax
    \dotfill\hbox to\@pnumwidth{\@tocpagenum{#7}}\par
    \nobreak
    \endgroup
  \fi}
\makeatother

\newenvironment{psmallmatrix}
  {\left(\begin{smallmatrix}}
  {\end{smallmatrix}\right)}

\numberwithin{equation}{subsection}
\newtheorem{thmx}{Theorem}

\newtheorem{innercustomthm}{Corollary}
\newenvironment{ccoro}[1]
  {\renewcommand\theinnercustomthm{#1}\innercustomthm}
  {\endinnercustomthm}
  
\newtheorem{theorem}[subsection]{Theorem}

\newtheorem{lemma}[subsection]{Lemma}

\newtheorem*{coron}{Corollary}

\newtheorem{prop}[subsection]{Proposition}

\theoremstyle{definition}
\newtheorem{defn}[subsection]{Definition}

\theoremstyle{remark}
\newtheorem{remark}[subsection]{Remark}
\newtheorem{exam}[subsubsection]{Example}
\newtheorem{nonexam}[subsubsection]{Non-example}
\newtheorem{almexam}[subsubsection]{Potential example}

\newcommand{\mZ}{\mathbf{Z}}
\newcommand{\mR}{\mathbf{R}}
\newcommand{\mQ}{\mathbf{Q}}
\newcommand{\mF}{\mathbf{F}}
\newcommand{\mC}{\mathbf{C}}
\newcommand{\mN}{\mathbf{N}}

\newcommand{\cO}{\mathcal{O}}

\newcommand{\fW}{\mathfrak{W}}

\newcommand{\ff}{\mathfrak{f}}

\newcommand{\fp}{\mathfrak{p}}

\newcommand{\fm}{\mathfrak{m}}
\newcommand{\fP}{\mathfrak{P}}

\newcommand{\fl}{\mathfrak{l}}
\newcommand{\fL}{\mathfrak{L}}

\newcommand{\fq}{\mathfrak{q}}
\newcommand{\fQ}{\mathfrak{Q}}

\newcommand{\lifaf}{\Longleftrightarrow}
\newcommand{\st}{\text{ such that }}

\newcommand{\nl}{\newline}

\newcommand{\rra}{\rightarrow}

\newcommand{\aad}{\text{ and }}

\newcommand{\ZZ}[1]{\mathbf{Z} / #1 \mathbf{Z}}
\newcommand{\brk}[1]{ \left\lbrace #1 \right\rbrace }
\newcommand{\brc}[1]{ \left[ #1 \right] }
\newcommand{\inv}{^{-1}}

\newcommand{\mof}{\text{ of }}

\newcommand{\mfor}{\text{ for }}
\newcommand{\mif}{\text{ if }}
\newcommand{\tth}{^{\text{th}}}

\newcommand{\gap}{\hspace{.5em}}

\newcommand{\iso}{\cong}

\newcommand{\pwr}[1]{\left( #1 \right)}

\newcommand{\lrra}{\longrightarrow}

\newcommand{\unit}{^{\times}}

\newcommand{\frob}{\text{Frob}}
\newcommand{\uZZ}[1]{\left( \mathbf{Z} / #1 \mathbf{Z} \right)^{\times}}

\def\quotient#1#2{\raise1ex\hbox{$#1$}{\Large/} \lower1ex\hbox{$#2$}}

\newcommand{\Kcomment}[1]{}
\newcommand{\cdef}[1]{{\color{black}\textbf{\textsf{#1}}}}

\DeclareMathOperator{\mwhere}{where}
\DeclareMathOperator{\gal}{Gal}
\DeclareMathOperator{\id}{id}

\DeclareMathOperator{\Frob}{Frob}

\DeclareMathOperator{\GL}{GL}

\DeclareMathOperator{\Hom}{Hom}

\DeclareMathOperator{\Gal}{Gal}

\DeclareMathOperator{\Primes}{\textsf{Primes}}

\DeclareMathOperator{\Sel}{Sel}

\DeclareMathOperator{\ord}{ord}

\DeclareMathOperator{\cl}{cl}
\DeclareMathOperator{\cH}{H}
\DeclareMathOperator{\res}{res}

\begin{document}

\title{Selmer groups of twists of elliptic curves over $K$ with $K$-rational torsion points}
\author{Jackson S. Morrow}
\address{Department of Mathematics and Computer Science, Emory University,
Atlanta, GA 30322}
\email{\href{mailto:jmorrow4692@gmail.com}{\texttt{jmorrow4692@gmail.com}}}
\maketitle
\vspace*{-2em}
\begin{abstract}
We generalize a result of Frey \cite[Theorem]{frey1987selmer} on Selmer groups of twists of elliptic curves over $\mathbf{Q}$ with $\mathbf{Q}$-rational torsion points to elliptic curves defined over number fields of small degree $K$ with a $K$-rational point. We also provide examples of elliptic curves coming from \cite{zywinapossible} that satisfy the conditions of our Corollary \ref{D}.
\end{abstract}
\vspace*{-1em}
\doublespacing
\tableofcontents
\singlespacing
\vspace*{-2em}
\section{Introduction}
Let $\ell $ be an odd, rational prime and let $E/K$ be an elliptic curve defined over a number field $K$. The $K$-rational points $E(K)$ form a finitely generated group by the Mordell\--Weil theorem. Recall from \cite[Section~X.4]{silvermanAEC} that we have the following exact sequence
$$0 \rra E(K)/\ell E(K) \rra \Sel_{\ell}(E,K) \rra \Sha(E,K)[\ell] \rra 0,$$
where $\Sel_{\ell}(E,K)$ denotes the \cdef{$\ell$-Selmer group} and $\Sha(E,K)[\ell]$ is \cdef{$\ell$-Shafarevich\--Tate group}. If $K = \mQ$, then Frey \cite{frey1987selmer} provides explicit examples of quadratic twist of elliptic curves over $\mQ$ with $\mQ$-rational points of odd, prime order $\ell$ whose $\ell$-Selmer groups are non-trivial; a theorem of Mazur \cite{mazur1977rational} implies that $\ell \in \brk{3,5,7}$.

\begin{theorem}[\cite{frey1987selmer}]\label{1.1}
Suppose that $E/\mQ$ is an elliptic curve with a $\mQ$-rational torsion point $P$ of odd prime order $\ell$, and suppose that $P$ is not contained in the kernel of reduction modulo $\ell$; in particular, this means that $E$ is not supersingular modulo $\ell$ if $\ord_{\ell}(j_E) \geq 0$. Let $\widetilde{S}_E$ be the subset of odd primes dividing the conductor $N(E)$ of $E$ defined by
\begin{align*}
\widetilde{S}_E &:= \brk{p | N(E) :  p \equiv -1 \pmod \ell ,\, \ell \nmid \ord_p(\Delta_E)}, 
\end{align*}
where $j_E$ is the $j$-invariant of $E$ and $\Delta_E$ is the discriminant of $E$. Suppose that $\widetilde{S}_E = \emptyset$. Suppose that $d\equiv 3 \pmod 4$ is a negative, square-free integer coprime to $\ell N(E)$ satisfying:
\begin{enumerate}
\item if $\ord_{\ell}(j_E) < 0$, then $\pwr{\frac{d}{\ell}} = -1$;
\item if $p|N(E)$ is an odd prime, then 
$$\pwr{\frac{d}{p}} = \left\lbrace \begin{array}{lclc}
-1 & &\mif \ord_p(j_E) \geq 0; \\
-1 & & \mif\ord_p(j_E) < 0 \aad E/\mQ_p \text{ is a Tate curve}; \\
1 & & \text{ otherwise}.
\end{array}\right.$$
\end{enumerate}
Then we have that $\Sel_{\ell}(E^d,\mQ)$ is non-trivial if and only if $\ell$-torsion of the class group of $\mQ(\sqrt{d})$ is non-trivial.
\end{theorem}
\begin{remark}
Frey actually proved a more explicit double divisibility statement \cite[Theorem]{frey1987selmer} concerning the $\ell$-Selmer group of $E^d$ and $\ell$-torsion of ray class groups, when $\widetilde{S}_E \neq \emptyset$; we completely generalize his double divisibility in Theorem \ref{C}. 
\end{remark}

Frey's idea was to obtain information about $\Sel_{\ell}(E^d,\mQ)$ when $E(\mQ)$ contains an element of order $\ell$. In particular, he studied the behavior of $E$ over local fields $\mQ_{\ell}$ and their algebraic closures $\overline{\mQ}_{\ell}$. His work illustrated to a deep relationship between $\ell$-ranks of Selmer groups and class groups of finite Galois extensions of exponent $\ell$. In this paper, we investigate the $\ell$-Selmer rank in families of quadratic twist of elliptic curves $E/K$ with $K$-rational points of odd prime order $\ell$. We use Frey's proof as a blueprint for our own, but the techniques we utilize come from class field theory. That being said, many of his arguments go through undisturbed.

In order to state our results, we first need to recall some facts concerning prime torsion of elliptic curves defined over number fields of small degree. We give a succinct summary of these results and refer the reader to \cite{sutherland2012torsion} for a more detailed synopsis. Let $S(n)$ denote the set of primes that can arise as the order of a rational point on an elliptic curve defined over a number field of degree $n$ and let \textsf{Primes}$(n)$ denote the set of primes bounded by $n$. By Merel\--Oesterl\'e's bound, we know that 
$$S(n) \subseteq \Primes((3^{n/2} + 1)^2).$$
The exact value of the set $S(n)$ is currently known for $n\leq 5$, but reasonable good bounds on $S(6)$ and $S(7)$ are given in \cite{derickx2012torsion}.
\begin{center}
\begin{tabular}{c|c|c}
$n$ & $S(n)$ & Reference \\
\hline
1 & $\Primes(7)$ & \cite{mazur1977rational} \\
2 & $\Primes(13)$ & \cite{kamienny1992torsion} \\
3 & $\Primes(13)$ &\cite{parent2003no} \\
4 & $\Primes(17)$ & \cite{kamienny2011torsion} \\
5 & $\Primes(19)$ & \cite{derickxtorsion}  \\
6 & $\subseteq \Primes(19) \cup \brk{37,73}$ & \cite{derickx2012torsion} 
\end{tabular}
\end{center}
One can also consider the subset $S_{\mQ}(n)\subseteq S(n)$ corresponding to primes that can arise as the order of a rational point on an elliptic curve $E_K = E \times_{\mQ} K$ where $E$ is defined over $\mQ$ and $K$ is a number field of degree $n$. From \cite{alvarotorsion2013}, it is known that 
$$S_{\mQ}(n) \subseteq \Primes(13) \cup \brk{37} \cup \Primes(2n+1),$$
and \cite[Corollary~1.1]{alvarotorsion2013} states that for $1\leq n \leq 20$, 
$$
S_{\mQ}(n) = \left\lbrace
\begin{array}{cc}
\Primes(7) &\mfor n= 1,2 \\
\brk{2,3,5,7,13}& \mfor n = 3,4 \\
\Primes(13) &  \mfor n = 5,6,7 \\
\Primes(17) & \mfor n = 8\\
\Primes(19) & \mfor n = 9,10,11 \\
\Primes(19) \cup {37} &\mfor 12\leq n \leq 20. \\
\end{array}
\right.
$$

In this paper, we generalize the full double divisibility statement of \cite[Theorem]{frey1987selmer} to elliptic curves defined over small degree number fields $K$. We state explicit versions of our results, Theorems \ref{1.3},\ref{C} and Corollary \ref{E}, in Section \ref{S32} once we have established some notation.

\subsection*{Some remarks about the proofs}
The problem of constructing elements in the Selmer group is a classical question with many avenues of approach. Frey's condition that the elliptic curve $E/K$ have a $K$-rational point of odd prime power order $\ell>3$ has two immediate consequences. First, the image of Galois under the mod $\ell$ representation is conjugate to 
$$
\begin{psmallmatrix} 1 & * \\ 0 & *\end{psmallmatrix} \subset \GL_2(\mF_{\ell}),
$$
which will assist in our explicitly description the Galois structure of splitting fields of $\ell$-covers of $E/K$ and the splitting fields of elements in $\Sel_{\ell}(E^d,K)$. The second is that we can immediately identify a quotient of $\cH^1(\Gal(\overline{K}/K),E(\overline{K})[\ell])$, namely $\cH^1(\Gal(\overline{K}/K),\mu_{\ell})$. Frey's (and our) proof relies on an analysis of cocyles in $\cH^1(\Gal(\overline{K}/K),E(\overline{K})[\ell])$ and this fact will allow us to deduce local triviality in certain cases using Hilbert's Theorem 90. A laborious aspect of our proofs is the case by case analysis of how primes $\fp $ dividing $N(E)$ behave in the field $K(\sqrt{d})\cdot K(E[\ell])$ where $d \in \cO_K\unit/(\cO_K\unit)^2$ yields the quadratic twist $E^d$ of $E$ and $K(E[\ell])$ is the $\ell$-division field of $E/K$. 

\subsection*{Organization of paper} In Section \ref{S2}, we recall some classical facts from class field theory and algebraic number theory. In Section \ref{S32}, we state our main results, Theorems \ref{1.3},\ref{C} and Corollary \ref{E}. In Section \ref{S3}, we prove Theorem \ref{1.3}, which yields a single divisibility statement. In Section \ref{S4}, we prove the double divisibility statement of Theorem \ref{C} by investigating the Galois structure of splitting fields of $\ell$-covers of $E/K$ and the splitting fields of elements $\Sel_{\ell}(E^d,K)$. Finally in Section \ref{S6}, we provide explicit examples of elliptic curves over $\mQ$ coming from \cite{zywinapossible} that satisfy the Corollary \ref{D}. 

\subsection*{Acknowledgments}The author wishes to thank Ken Ono for initially suggesting this project, David Zureick-Brown for his guidance and patience in explaining the finer details and for his help generalizing the conditions of \cite{frey1987selmer} in a series of conversations, \cite{stackexchange} for help in Lemma \ref{3.1}, and \cite{mathoverflow} for help in Remark \ref{2.4}. The computations in this paper were performed using the \textsc{Magma} computer algebra system \cite{MR1484478}. For \textsc{Magma} code verifying the claims in Section \ref{S6}, we refer the reader to \cite{morrowSelmer2015}.


\section{Background \& Notation}\label{S2}
Let $L/K$ be a Galois extension of $K$, with ring of integers $\cO_L \aad \cO_K.$ For any finite prime $\fP \in \cO_L$ lying over a prime $\fp \in \cO_K,$ let $D(\fP)$ denote the \cdef{decomposition group of $\fP ,$} let $I(\fP)$ denote \cdef{the inertia group of $\fP$} and let $\kappa' := \cO_L / \fP \aad \kappa= \cO_K / \fp$ be the residue fields of characteristic $q = p^n$. The Galois theory of the extension encodes the splitting and ramification of $\fP$ over $\fp$, in particular, we have the below correspondence
\begin{figure*}[h!]
$$
\xymatrix{
L   & {1} \ar@{-}[d]^{e=|I(\fP)|} \\
K^{I(\fP)}\ar@{-}[u]^{e=|I(\fP)|} & I(\fP) \ar@{-}[d]^{f=|D(\fP)|/e} \\
K^{D(\fP)} \ar@{-}[u]^{f=|D(\fP)|/e} & D(\fP) \ar@{-}[d]^{g = n/ef}\\
K \ar@{-}[u]^{g = n/ef} & \Gal (L/K) 
}
$$
\end{figure*}

The exact sequence
\begin{equation*}
1 \lrra I(\fP) \lrra D(\fP) \lrra \Gal (\kappa'/\kappa) \lrra 1
\end{equation*}
induces an isomorphism $D(\fP)/ I(\fP) \iso \Gal(\kappa'/\kappa).$ In particular, there is a unique element in $D(\fP) / I(\fP),$ denote by $\brc{\frac{L/K}{\fP}},$ which maps to the $q\tth$ power Frobenius map $\Frob_q \in \Gal (\kappa'/\kappa)$ under the isomorphism, where $q$ is the number of elements in $\kappa.$ The notation $\brc{\frac{L/K}{\fP}}$ is referred to as the \cdef{Artin symbol} of the extension $L/K$ at $\fP$. If $L/K$ is an abelian extension, then the Frobenius automorphism $\brc{\frac{L/K}{\fP}}$ is denoted $\pwr{\frac{L/K}{\fp}}$; this change in notation reflects the fact that the automorphism is determined by $\fp \in \cO_K$ independent of the primes $\fP \mof \cO_L$ above it.

\begin{defn}\label{2.1}
Now let $\fm$ be a modulus divisible by all (finite or infinite) ramified primes of an abelian extension  $L/K.$ There is therefore a canonically defined Frobenius element in $\gal (L/K)$ denoted $\frob_{\fp}$. The \cdef{Artin symbol} of $L/K$ is defined on the group of prime-to-$\fm$ fractional ideals, $I_{K}(\fm)$, by linearity:
\begin{eqnarray*}
\pwr{\frac{L/K}{\bullet}}\colon I_{K}(\fm) &\lrra &  \gal (L/K) \\
 \prod_{i=1}^m \fp_i^{n_i} &\longmapsto &\prod_{i=1}^m\frob_{\fp_i}^{n_i}.
\end{eqnarray*}
Therefore, we can extend the Artin symbol to give us a group homomorphism
\begin{equation*}
\Phi_{\fm}\colon I_K(\fm) \lrra  \gal (L/K)
\end{equation*}
called the \cdef{global Artin map}.
\end{defn}
In this note, we need a specific result concerning the Artin symbol and ramification theory for quadratic extensions $L/K$. 
\begin{lemma}\label{2.15}
Let $L/K$ be a quadratic extension, let $\fp$ be a prime ideal of $\cO_K$, let $\fm = \Delta_{L/K}$ in the definition of the global Artin map, and let $\fP$ denote some prime of $\cO_L$ lying above $\fp$, and let $\langle \delta \rangle = \Gal(L/K).$ Then:
\begin{enumerate}
\item $\fp$ is unramified and splits completely in $L$ $\lifaf$ $\pwr{\frac{L/K}{\fp}} = \id$,
\item $\fp$ is unramified and non-split in $L$ $\lifaf$ $\pwr{\frac{L/K}{\fp}} = \delta$,
\item $\fp$ is ramified in $L$ $\lifaf$ $\fp | \Delta_{L/K}$ where $\Delta_{L/K}$ denotes the relative discriminant of $L/K$.
\end{enumerate}
\end{lemma}  
\begin{proof}
Part (3) follows from Definition \ref{2.1}. Since $\fp$ is unramified, we know that $|D(\fP)| = f$ where $f$ is the inertia degree of $\fP$ over $\fp$.  A prime $\fp$ splits completely in $L$ if and only if the ramification index $e$ of $\fP$ above $\fp$ and the inertia degree $f$ of $\fP$ above $\fp$ are equal to $1$. Hence, 
$$|D(\fP)| = [\kappa' : \kappa] = 1 \lifaf \ord \pwr{\frac{L/K}{\fp}} = 1 \lifaf \pwr{\frac{L/K}{\fp}}  = \id ,$$
which proves (1). For (2), our assumptions and the fundamental identity tell us that $e = 1$ and $g =1$ if and only if $f = 2$. Thus, 
$$|D(\fP)| = [\kappa' : \kappa] = 2 \lifaf \ord \pwr{\frac{L/K}{\fp}} = 2 \lifaf \pwr{\frac{L/K}{\fp}}  = \delta .$$
\end{proof}

In Theorem \ref{1.3}, we use Hecke characters to describe a subset of primes $\fp | N(E)$. We recall the definition of these characters and discuss how their values can encode information about ramification.
\begin{defn}
Let $\ff$ be a non-zero ideal of $\cO_K$, and let 
$$\chi_{\infty}\colon (\mR\unit)^{r_1} \times (\mC\unit)^{r_2} \lrra \mC\unit$$
be a continuous character where $[K:\mQ] = r_1 + 2r_2$. Then the character 
$$\chi_{\cH}\colon I(\ff) \lrra \mC\unit$$
is a \cdef{Hecke character} with \cdef{conductor $\ff$} and \cdef{infinity-type $\chi_{\infty}$} if the following diagram commutes:
$$
\xymatrixrowsep{.2in}
\xymatrixcolsep{.2in}
\xymatrix{
{} & P(\ff) \ar@{->}[rd]^{\chi_{\cH}}& {} \\
K_{\ff} \ar@{->}[ru]^{\alpha \mapsto (\alpha)} \ar@{->}[rd]^{\alpha \mapsto 1\otimes \alpha} & {} & \mC\unit \\
{} & (\mR\unit)^{r_1} \times (\mC\unit)^{r_2} \ar@{->}[ru]^{\chi_{\infty}} & {}
}
$$
where $I(\ff)$ is the group of fractional ideals corprime to $\ff$ and $P(\ff)$ is the group of principal ideals of $\cO_K$ relatively prime to $\ff$. A Hecke character is \cdef{primitive} if it is not induced from another classical Hecke character with conductor $\ff'|\ff$. 
\end{defn}
\begin{remark}\label{2.4}
Recall that there is a conductor-preserving correspondence between primitive Dirichlet characters of order $\ell$ and cyclic, degree $\ell$ number fields $k/\mQ$. From \cite[Theorem~3.7]{washington2012introduction},  the Dirichlet character $\chi$ corresponds to the fixed field $k$ of $\ker \chi \subseteq \uZZ{f_{\chi}} = \Gal(\mQ(\zeta_{f_{\chi}})/\mQ)$. For any prime $q$, 
$$\chi(q) = 0 \lifaf q \text{ ramifies in }k, \quad \aad \quad \chi(q) = 1 \lifaf q \text{ splits in }k.$$

By class field theory, any Hecke character $\chi_{\cH}$ of $K$ of order $\ell$ determines a cyclic extension $N/K$ of degree $\ell$. Moreover, the set of Hecke characters determining this cyclic extension equals $\{\chi_{\cH},\chi_{\cH}^2,\dots , \chi_{\cH}^{\ell - 1}\}$. These $\ell - 1$ Hecke characters have the same conductor $\ff$, and the determinant of $L/K$ equals their product $\ff^{\ell - 1}$ by the Hasse conductor-discriminant theorem. Thus for any prime ideal $\fq$ of $\cO_K$, we have that
$$\chi_{\cH}(\fq) = 0 \lifaf \fq \text{ ramifies in }N, \quad \aad \quad \chi_{\cH}(\fq) = 1 \lifaf \fq \text{ splits in }N.$$
\end{remark}

\subsubsection*{Notation} We set the following notation.
\begin{align*}
K &:=\text{ Galois number field,}\\
\ell &:= \text{ odd, rational prime in } S(n)\setminus\brk{2,3} \st \ell \nmid \cl(K) \aad \zeta_{\ell} \notin K, \\
L/K &:= \text{ algebraic extension of }K, \\
\fp &:= \text{ prime divisor of the rational prime }p \text{ in } \cO_K,\\
\fP &:= \text{ prime divisor of }\fp \text{ in } \cO_L,\\
K_{\fp} &:= \text{ completion of }K \text{ with respect to }\fp ,\\
L_{\fP} &:= \text{ completion of }L \text{ with respect to }\fP ,\\
S &:= \text{ finite set of primes of }\cO_K ,\\
M/L &:= \text{ Galois extension with abelian Galois group of exponent }\ell.
\end{align*}
More generally, lower case gothic font will denote a divisor of a rational prime of $\mQ$, and similarly, upper case gothic font will denote a divisor of a prime of $K$.

\begin{defn}
$M/L$ is said to be \cdef{little ramified outside $S$} if for primes $\fp\notin S$ and all $\fP_L|\fp$ one has
$$M\cdot L_{\fP}(\zeta_{\ell}) = L_{\fP}(\zeta_{\ell})(\sqrt[{\ell}]{u_1},\dots ,\sqrt[{\ell}]{u_k})$$
with $k\in \mN$ and $\ord_{\fP_L}(u_i) = 0$. Here $\zeta_{\ell}$ is a ${\ell}\tth$ root of unity, $u_1,\dots , u_k$ are elements in $L_{\fP}(\zeta_{\ell})$, and $\ord_{\fP_L}$ is the normed valuation belonging to $\fP_L$.
\end{defn}
If $M/L$ little ramified outside $S$, then $M/L$ is unramified at all divisors of primes $\fp\notin S \cup \brk{\fl}$.
\vspace*{-.2em}
\subsubsection*{Notation}We set the following notation, which comes directly from \cite{frey1987selmer}:
\begin{align*}
L_S &:= \text{ maximal abelian extension of exponent $\ell$ of $L$ which is}  \\ 
& \hspace*{1.7em}\text{ little ramified outside $S$}, \\
L_{S,u} &:= \text{ maximal subfield of $L_{S}$ which is unramified outside of $S$}, \\ 
H_S(L) &:= \text{  Galois group of $L_{S}/L$}, \\ 
H_{S,u}(L) &:= \text{ Galois group of $L_{S,u}/L$} ,\\
\cl_{S}(L)\brc{\ell} &:= \text{ order of $H_{S}(L)$}, \\
\cl_{S,u}(L)\brc{\ell} &:= \text{ order of $H_{S,u}(L)$}.
\end{align*}

\begin{remark}\label{2.6}
If $S = \emptyset$, we see that $\cl_{\emptyset,u}(L)$ is equal to the order of the subgroup of the divisor class group of $L$ consisting of elements of order $\ell$ which we denote by $\cl(L)[\ell]$. 
\end{remark}

Now assume that $L/K$ is normal with cyclic Galois group generated by an element $\gamma$ of order $\ell-1$. Take an extension $\widetilde{\gamma}$ to $L(\zeta_{\ell})$. Let $\chi_{\ell}$ be the cyclotomic character induced by the action of $\Gal(L(\zeta_{\ell})/K)$ on $\langle \zeta_{\ell} \rangle$. Then $\chi_{\ell}(\widetilde{\gamma})$ is determined by 
$$\widetilde{\gamma}(\zeta_{\ell}) = \zeta_{\ell}^{\chi_{\ell}(\widetilde{\gamma})}.$$
Let $M$ be normal over $K$ containing $L$ such that $\Gal(M/L)$ is abelian of exponent $\ell$. Then $\widetilde{\gamma}$ operates by conjugation on $$\Gal(M(\zeta_{\ell})/L(\zeta_{\ell})) \iso \Gal(M/L), $$
and this operation does not depend on choice of $\widetilde{\gamma}$. Hence the subgroup 
$$H(\chi_{\ell}):= \brk{\alpha \in \Gal(M/L) : \widetilde{\gamma}\alpha \widetilde{\gamma}\inv = \alpha^{\chi_{\ell}(\widetilde{\gamma})}} \subseteq \Gal (M/L)$$
is well-defined. In the special case that $M = L_S$, we denote the order of $H_{S}(L)(\chi_{\ell})$ by $\cl_S(L)_{\ell}(\chi_{\ell})$.

Now we shall consider an elliptic curve $E/K$ given by a Weierstrass equation $F(x,y) = 0$ with coefficients in $\cO_K$ and minimal discriminant $\Delta_E$. For any extension $L/K$, we denote the $L$-rational points of $E$ (including $\infty$) by $E(L)$. Let $\chi_{\cH}$ be a primitive Hecke character of order $\ell$ and let 
\begin{align*}
\widetilde{S}_E &:= \brk{\fp | N(E) : \chi_{\cH}(\fp) \neq 0 ,\ord_{\fp}(\Delta_E)\not\equiv 0 \pmod \ell} \\
S_E &:= \{\fp \in \widetilde{S}_E : \ord_{\fp}(j_E) < 0 \}.
\end{align*}
Let $d \in \cO_K\unit/(\cO_K\unit)^2$ and denote the twist of $E/K$ by $E^d/K$. Via the general theory of twists \cite[Section~X.2]{silvermanAEC}, we know that $E^d$ is isomorphic to $E$ over $K(\sqrt{d})$ but not over $K$. Let $G_K := \Gal(\overline{K},K)$ denote the absolute Galois group. Let $\fW(E^d,K)[\ell]$ be the set of elements of order $\ell$ in the kernel of 
$$\rho\colon \cH^1(G_K,E^d(\overline{K})) \lrra \bigoplus_{\fp \text{ prime}} \cH^1(\Gal(\overline{K_{\fp}}/K_{\fp}),E^d(\overline{K_{\fp}})).$$
The group of elements of order $\ell$ in the Selmer group of $E^d$, denoted by $\Sel_{\ell}(E^d,K)$ is given as the pre-image of $\fW(E^d,K)[\ell]$ by the map 
$$\alpha \colon \cH^1(G_K,E^d(\overline{K})[\ell]) \lrra \cH^1(G_K,E^d(\overline{K})).$$
There are two main cases we need to consider:
\subsubsection*{Case 1:} Assume that $\ord_{\fp}(j_E) \geq 0$. Then there is a finite extension $N/K$ such that $E$ has good reduction modulo all $\fP_N | \fp$ i.e., we find an elliptic curve $\widetilde{E}$ such that $\tilde{E}$ modulo $\fP_N$ is an elliptic curve over the residue field of $\fP_N$. $\widetilde{E}(\overline{N_{\fP}})$ contains a subgroup $\widetilde{E}_-(N_{\fP})$ consisting of points $(\widetilde{x},\widetilde{y})$ with $\ord_{\fP_N}(\widetilde{x}) < 0$. $\widetilde{E}_-$ is the \cdef{kernel of reduction modulo $\fP_N$}, and $\ord_{\fP_N}(\widetilde{x}/\widetilde{y})$ is the \cdef{level} of $(\widetilde{x},\widetilde{y})$. For ease of notation, we say that a point $(x,y) \in E(\overline{N_{\fP}})$ is in the kernel of the reduction modulo $\fP_N$ if its image $(\widetilde{x},\widetilde{y}) \in \widetilde{E}_-(\overline{N_{\fP}})$.
\subsubsection*{Case 2:} Assume that $\ord_{\fp}(j_E) < 0$. Then after an extension $L/K_{\fp}$  of degree $\leq 2$, $E$ becomes a Tate curve (via a theorem of Tate \cite[Theorem~C.14.1]{silvermanAEC}); in particular, one has a Tate parametrization 
$$\tau \colon \overline{L}\unit/\langle q \rangle \lrra E(\overline{L})$$
where $q$ is the $\fp$-adic period of $E$. One also has that $$j_E = \frac{1}{q} + \sum_{i=0}^{\infty} a_iq^i \quad \text{ with }a_i \in \mZ$$
and the points of order $\ell$ in $E(\overline{L})$ are of the form $\tau(\zeta_{\ell}^{\alpha}(q^{\beta/\ell}))$ where $\alpha,\beta \in \brk{1,\dots , \ell-1}$. 
\begin{defn}
If $F/K$ is a number field and $\fP_F | \fp$ we say that a point $(x,y) \in E(F_{\fP})$ is in the \cdef{connected component of the unity modulo $\fP_F$} if it is of the form $\tau(u)$ with $u$ a $\fP_F$-adic unit, and $(x,y)$ is in the kernel of the reduction modulo $\fP_F$ if $u-1 \in \fP_F$. 
\end{defn}
\begin{remark}\label{2.19}
One should notice that if $E$ is not a Tate curve over $K_{\fp}$ but over an extension of degree $2$ of $K_{\fp}$, then for all points $P \in E(K_{\fp})$, $2P$ is in the connected component of unity modulo $\fp$.
\end{remark}

\section{Statement of Results}\label{S32}
As mentioned above, \cite[Theorem]{frey1987selmer} gives a double divsibility statement involving the $\ell$-torsion of the Selmer group. First, we generalize his single divisibility to elliptic curves $E/K$ defined over number fields $K$ of finite degree with $K$-rational points of odd, prime order $\ell$. Recall that $S(n)$ is the set of primes that can arise as the order of a rational point on an elliptic curve defined over a number field of degree $n$.
\begin{thmx}\label{1.3}
Let $K$ be a Galois number field and choose $\ell \in S(n)\setminus\brk{2,3} $ such that $\ell \nmid \cl(K)$ and $\zeta_{\ell} \notin K $.
Let $E/K$ be an elliptic curve over $K$ with a $K$-rational point $P$ of order $\ell$; let $\chi_{\cH}$ denote a primitive Hecke character of $K $ with order $\ell$; let $\fq$ denote a prime of $\cO_K$ that lies above 2; and let $\fl$ denote a prime of $\cO_K$ that lies above $\ell$. Suppose that $P$ is not contained in the kernel of reduction modulo $\fl$; in particular, this means that $E$ is not supersingular modulo $\fl$ if $\ord_{\fl}(j_E) \geq 0$. Let $S_E$ be the set of primes
\begin{align*}
S_E &:= \brk{\fp | N(E) : \ord_{\fp}(\Delta_{E}) \not\equiv 0 \pmod \ell , \, \chi_{\cH}(\fp) \neq 0, \aad \ord_{\fp}(j_{E}) < 0}.
\end{align*}
Suppose that $d \in \cO_K\unit/(\cO_K^{\times})^2$ is negative\footnote{We say that $d \in \cO_K\unit/(\cO_K^{\times})^2$ is negative if the image of $d$ under each real embedding is negative}, coprime to $\fl \cdot  N(E)$, and satisfies the following divisibility and Artin symbol conditions where $\langle \delta \rangle = \Gal(K(\sqrt{d})/K)$:
\begin{enumerate}
\item if $\fq | N(E)$, then $\fq | \Delta_{K(\sqrt{d})/K}$;
\item if $\ord_{\fl}(j_{E}) < 0$, then $\pwr{\frac{K(\sqrt{d})/K}{\fl}} = \delta$;
\item if $\fp|N(E)$ is a prime of $K$ with $\fp\notin S_E$, then 
\begin{enumerate}
\item[$\bullet$] if $\ord_{\fp}(j_{E}) \geq 0$, then $\pwr{\frac{K(\sqrt{d})/K}{\fp}} = \delta$;
\item[$\bullet$] if $\ord_{\fp}(j_{E}) < 0 \aad E/K_{\fp} \text{ is a Tate curve},$ then $\pwr{\frac{K(\sqrt{d})/K}{\fp}} = \delta$;
\item[$\bullet$] otherwise, $\pwr{\frac{K(\sqrt{d})/K}{\fp}} = \id$.
\end{enumerate}
\end{enumerate}
Then we have that the order of the $\ell$-torsion of the $S_E$-ray class group of $K(\sqrt{d})$ divides the order of $\Sel_{\ell}(E^d,K)$. More precisely, the single divisibility statement holds:
\begin{equation}\label{eqn2}
\cl_{S_E,u}(K(\sqrt{d}))[\ell] \Big{|} \# \Sel_{\ell}(E^d,K).
\end{equation}
\end{thmx}

We also prove a stronger, more explicit version of Theorem \ref{1.3} in the form of a double divisibility statement, which completely generalizes \cite[Theorem]{frey1987selmer}. 
\begin{thmx}\label{C}
Let $K$ be a Galois number field of degree $n\leq 5$ such that $N_{K/\mQ}(\fq) = 2   \text{ for all }\fq|2$. Choose $\ell \in S(n)\setminus\brk{2,3} $ such that $\ell \nmid \cl(K)$ and $\zeta_{\ell} \notin K $. Let $E/K$ be an elliptic curve over $K$ with a $K$-rational point $P$ of order $\ell$; let $\chi_{\cH}$ denote a primitive Hecke character of $K $ with order $\ell$; let $\fq$ denote a prime ideal of $\cO_K$ that lies above 2; and let $\fl$ denote a prime ideal of $\cO_K$ that lies above $\ell$. If $[K:\mQ] = 5$ and $\ell = 5$, then we must make the added assumption that $(\ell)\cO_K$ is not totally ramified. Suppose that $P$ is not contained in the kernel of reduction modulo $\fl$; in particular, this means that $E$ is not supersingular modulo $\fl$ if $\ord_{\fl}(j_E) \geq 0$. Let $\widetilde{S}_E$ and $S_E$ be the sets of primes
\begin{align*}
\widetilde{S}_E &:= \brk{\fp | N(E) :\chi_{\cH}(\fp) \neq 0,\, \ord_{\fp}(\Delta_{E}) \not\equiv 0 \pmod \ell },\\
S_E &:= \{\fp \in \widetilde{S}_E : \ord_{\fp}(j_{E}) < 0\}.
\end{align*}
Suppose that $d \in \cO_K\unit/(\cO_K^{\times})^2$ is negative, coprime to $\fl \cdot  N(E)$, and satisfies the following divisibility and Artin symbol conditions where $\langle \delta \rangle = \Gal(K(\sqrt{d})/K)$:
\begin{enumerate}
\item if $\fq | N(E)$, then $\fq | \Delta_{K(\sqrt{d})/K}$;
\item if $\ord_{\fl}(j_{E}) < 0$, then $\pwr{\frac{K(\sqrt{d})/K}{\fl}} = \delta$;
\item if $\fp|N(E)$ is a prime of $K$ with $\fp\notin S_E$, then 
\begin{enumerate}
\item[$\bullet$] if $\ord_{\fp}(j_{E}) \geq 0$, then $\pwr{\frac{K(\sqrt{d})/K}{\fp}} = \delta$;
\item[$\bullet$] if $\ord_{\fp}(j_{E}) < 0 \aad E/K_{\fp} \text{ is a Tate curve},$ then $\pwr{\frac{K(\sqrt{d})/K}{\fp}} = \delta$;
\item[$\bullet$] otherwise, $\pwr{\frac{K(\sqrt{d})/K}{\fp}} = \id$.
\end{enumerate}
\end{enumerate}
Then we have the following double divisibility
\begin{equation}\label{eqn1}
\cl_{S_E,u}(K(\sqrt{d}))[\ell] \Big{|} \# \Sel_{\ell}(E^d,K) \Big{|} \cl_{\widetilde{S}_E,u}(K(\sqrt{d}))[\ell]\cdot \cl_{S_E}(K')[\ell](\chi_{\ell}),
\end{equation}
where $K'$ is the subfield of $K(\sqrt{d},\zeta_{\ell})$ of index 2 containing neither $\zeta_{\ell}$ nor $\sqrt{d}$. 
\end{thmx}
\begin{remark}
In words, \eqref{eqn1} states that the order of the $\ell$-torsion of the $S_E$-ray class group of $K(\sqrt{d})$ divides the order of $\Sel_{\ell}(E^d,K)$, and the order of $\Sel_{\ell}(E^d,K)$ divides the order of the $\ell$-torsion of the $\widetilde{S}_E$-ray class group of $K(\sqrt{d})$ times the degree of the maximal abelian extension $K''$ of $K'$ of exponent $\ell$ unramified outside of $S_E \cup \brk{\fl}$ such that the Galois group $\Gal(K'/K)$ acts on $\Gal(K''/K)$ by $\chi_{\ell}\varepsilon_d$, where $\varepsilon_d$ is the character prescribing the Galois action on $\sqrt{d}$.
\end{remark}
Once we have proved Theorems \ref{1.3}, \ref{C}, we can immediately extend the divisibility statements \eqref{eqn2}, \eqref{eqn1} to elliptic curves $E$ defined over $\mQ$ by considering the values of $S_{\mQ}(n)$.
\begin{ccoro}{C}\label{B}
Let $E/\mQ$ be an elliptic curve defined over $\mQ$. For some Galois number field $K$, suppose that $E_K$ attains a $K$-rational point $P$ of order $\ell $ where $\ell \in S_{\mQ}(n)\setminus\brk{2,3} $ such that $\ell \nmid \cl(K)$ and $\zeta_{\ell} \notin K $. In keeping with the notation and assumptions of Theorem \ref{1.3}, we can produce examples of quadratic twists $E_K^d $ that satisfy the divisibility statement \eqref{eqn2}.
\end{ccoro}

\begin{ccoro}{D}\label{D}
Let $E/\mQ$ be an elliptic curve defined over $\mQ$; let $E_K$ denote the base change of this curve to a Galois number field of degree $n\leq 20$ such that $N_{K/\mQ}(\fq) = 2   \text{ for all }\fq|2$. Choose $\ell \in S_{\mQ}(n)\setminus\brk{2,3} $ such that $\ell \nmid \cl(K)$, $\zeta_{\ell} \notin K $, and the ramification index $e_{\fl}(K/\mQ)$ satisfies $1>e_{\fl}(K/\mQ)/(\ell-1) - 1$. Suppose that $E_K$ attains a $K$-rational point $P$ of order $\ell $, then in keeping with the notation and assumptions of Theorem \ref{C}, we can produce examples of quadratic twists $E_K^d $ that satisfy the double divisibility statement \eqref{eqn1}.
\end{ccoro}

We can also generalize \cite[Corollary]{frey1987selmer}, which we stated as Theorem \ref{1.1}.
\begin{ccoro}{E}\label{E}
Let $(E,\, \ell,\,  K,\, d)$ be as in Theorem \ref{C} or in Corollary \ref{D}. If $\widetilde{S}_E = \emptyset$, then $\Sel_{\ell}(E^d,K)$ is non-trivial if and only if the $\ell$-torsion of the class group of $K(\sqrt{d})$ is non-trivial, in particular 
$$\cl(K(\sqrt{d}))[\ell]  \Big{|} \# \Sel_{\ell}(E^d,K)  \Big{|} (\cl(K(\sqrt{d}))[\ell])^2.$$
\end{ccoro}

\begin{remark}
In his Ph.D. thesis \cite{mailhot2003Thesis}, Mailhot was able to recover and sharpen \cite[Theorem]{frey1987selmer} for elliptic curves defined over $\mQ$ using purely cohomological methods. His refinement comes from prescribing a splitting behavior of primes above $K'$ instead of just a non-ramified condition. We remark that our methods and results are disjoint, however, we believe that \cite[Corollary~2.17]{mailhot2003Thesis} can be generalized to elliptic curves defined over number fields $K$, using Theorem \ref{C}.
\end{remark}

\section{Proof of Theorem \ref{1.3}}\label{S3}
In this section, we prove the divisibility statement \eqref{eqn2}. Before we proceed, we make a remark about some of the prime assumptions of Theorem \ref{1.3}.

\begin{remark}[Prime assumptions]
If $\ord_{\fp}(j_{E}) < 0$, then we have that $E/{K_{\fp}}$ has a Tate parametrization. The second condition $\ord_{\fp}(\Delta_{E}) \not\equiv 0 \pmod \ell$ assists us in Lemma \ref{4.3}. In short, it allows us to understand ramification in the $\ell$-division field of $E_{K_{\fp}}$. The final condition $\chi_{\cH}(\fp) \neq 0$ is used in Lemma \ref{4.4} and is an analogue of Frey's condition that $p \equiv -1 \pmod \ell$. Moreover, this condition allows us to deduce, using Remark \ref{2.4}, that for a cyclic extension $M_2/K$ of degree $\ell$, $\fp$ is unramified in $M_2$.
\end{remark}
The first step in the proof is to exhibit an element in $\Sel_{\ell}(E^d,K)$.
\begin{lemma}\label{4.3}
Let $\ell > 3$ be a rational prime; let $M/K$ be a non-abelian Galois extension of degree $2\ell$ containing $K(\sqrt{d})$ that is unramified over this field outside of $S_E $; let $\alpha$ be a generator of $\Gal(M/K(\sqrt{d}))$; and let $\phi$ the element in $\cH^1(\Gal(M/K),E^d(M)[\ell])$ determined by $\phi(\alpha) = P$, where $P$ is a $K$-rational point of order $\ell$. Then $\phi$ is an element of $\Sel_{\ell}(E^d,K)$.
\end{lemma}
\begin{proof}
First, we need to show that there exists some element $$\phi \in \cH^1(\Gal(M/K),E^d(M)[\ell])$$ whose restriction $\overline{\phi}$ to $\Gal(M/K(\sqrt{d})) = \langle \alpha \rangle$ is given by $\overline{\phi}(\alpha) = P$. We identify $E^d(M)[\ell]$ with $E(M)[\ell] = \langle P\rangle$. Since $E^d(K(\sqrt{d}))[\ell] = \langle P \rangle$ and $\delta (P) = -P$ where $\langle \delta \rangle = \Gal(K(\sqrt{d})/K)$, we have invariance of $\phi$ under $\delta$ from the fact that $\delta \alpha \delta  = \alpha\inv $. Since $$\cH^1(\Gal(M/K),E^d(M)[\ell]) = \cH^1(\Gal(M/K(\sqrt{d})),E^d(M)[\ell])^{\delta},$$
our assertions follows.

Hence it remains to show that $\overline{\phi}$ is locally trivial when regarded as an element of 
$$\cH^1(\Gal(M/K(\sqrt{d})),E^d(M)).$$
We may restrict ourselves to primes $\fP_M | \fl \cdot N(E)$. By condition (1) of Theorem \ref{1.3}, the divisors of $\fq$ are unramified in $M/K(\sqrt{d})$ if $\fq | N(E)$, and hence we may assume that $\fP_M \nmid \fq$.

Assume that $\pwr{\frac{K(\sqrt{d})/K}{\fp}} = \delta$. In this case, $\fP_M$ is either fully ramified or decomposed (since $M/K$ is non-abelian). So assume that $\fP_M$ is fully ramified and divides $\fp$. Then $\fp \in S_E$ and in particular $\fp \neq \fl$ and $\ord_{\fp}(\Delta_{E_K}) \neq 0 \pmod \ell$. We claim that $E^d/K_{\fp}(\sqrt{d})$ is a Tate curve and that $P$ is contained in the connected component of the unity over $K_{\fp}(\sqrt{d})$ corresponding to an $\ell\tth$ root of unity $\zeta_{\ell}$. 

The fact that $E^d/K_{\fp}(\sqrt{d})$ is a Tate curve follows since $\fp \in S_E$ and so $\ord_{\fp}(j_E) < 0$. Since $\ord_{\fp}(\Delta_{E}) \neq 0 \pmod \ell$, we know that adjoining $q^{1/\ell}$ to $K_{\fp}(\sqrt{d})$, where $q$ is the $\fp$-adic period of $E$, is a non-trivial extension. Under the Tate parametrization $\tau$, we have that torsion points of order $\ell$ in $E^d(\overline{K_{\fp}(\sqrt{d})})[\ell]$ are of the form $\tau(\zeta_{\ell}^{\alpha}q^{\beta/\ell})$ where $\alpha,\beta \in \brk{1,\dots , \ell - 1}$. Since $P$ is a point of order $\ell$ defined over $K_{\fp}(\sqrt{d})$, we know that $\zeta_{\ell}^{\alpha} \in K_{\fp}(\sqrt{d})$ for some $\alpha\in \brk{1,\dots , \ell-1}$ and that $$\tau\inv(P) = \zeta_{\ell}^{\alpha}q^{\beta/\ell} \in K_{\fp}(\sqrt{d}).$$
In order for $\zeta_{\ell}^{\alpha}q^{\beta/\ell} \in K_{\fp}(\sqrt{d})$, we must have that $\beta = 0$ since $q$ is not an $1/\ell\tth$ power. Thus, $\tau\inv(P) = \zeta_{\ell}^{\alpha}$, and hence $P$ is contained in the connected component of the unity over $K_{\fp}(\sqrt{d})$ corresponding to an $\ell\tth$ root of unity $\zeta_{\ell}$. Since $M_{\fP}/K_{\fp}(\sqrt{d})$ is cyclic of degree $\ell$, we have that $\zeta_{\ell} = \alpha x / x$ for some $x\in M_{\fP}$ by Hilbert' Theorem 90, and therefore, $\overline{\phi}$ is trivial when considered in $\cH^1(\Gal(M_{\fP}/K_{\fp}),E^d(M_{\fP}))$.


Next assume that $\pwr{\frac{K(\sqrt{d})/K}{\fp}} = \id$ and $\fp \neq \fl$. Then $\ord_{\fp}(j_{E}) < 0$ and $E$ is a Tate curve over $K_{\fp}$, and so again $P$ corresponds to some $\ell\tth$ root of unity $\zeta_{\ell}$ under the Tate parametrization of $E = E^d$ over $K_{\fp}(\zeta_{\ell})$ and hence $\overline{\phi}$ is split by $K_{\fp}(\zeta_{\ell})$ as seen above. But since the degree of $K_{\fp}(\zeta_{\ell})$ over $K_{\fp}$ is prime to $\ell$, $\overline{\phi}$ is split over $K_{\fp}$ already, and thus $\overline{\phi}$ is locally trivial.

There is one remaining case: $\fp = \fl$ and $\ord_{\fl}(j_{E}) \geq 0$. Let $\fL_M | \fl$. By the assumption, $M/K$ is unramified at $\fL_M$, and we can find a normal extension $N/K$ of degree prime to $\ell$ such that $E$ has good reduction modulo all primes $\fL_N | \fl$. In particular, we may take $N = K(\zeta_{12},\sqrt[12]{\fl})$. Now $$\cH^1(\Gal(M_{\fL}\cdot N/K_{\fl}\cdot N),E^d(M_{\fL}\cdot N)) = 0$$ since the reduction of $E^d$ modulo $\fL$ is good and $M_{\fL}N/K_{\fl}N$ is unramified, and hence it follows that 
$$\cH^1(\Gal(M_{\fL}/K_{\fl}),E^d(M_{\fL})) = 0.$$
\end{proof}
Next, we look at the action of $\delta$ on $H_{S_E,u}(K(\sqrt{d}))$.
\begin{lemma}\label{4.4}
The generator $\langle \delta \rangle = \Gal(K(\sqrt{d})/K)$ acts as $-\id$ on the Galois group $H_{S_E,u}(K(\sqrt{d}))$.
\end{lemma}
\begin{proof}
We may write
$$H_{S_E,u}(K(\sqrt{d})) = H^- \oplus H^+$$
where $H^-$ is the part where $\delta$ acts as $-\id$, and $H^+$ the part with $\delta = \id$. Let $\widetilde{M} := M_{S_E,u}^{H^-}$, which is the fixed field of $M_{S_E,u}$ by $H^-$. Assume that $M_1$ is a subfield of $\widetilde{M}$ that is cyclic over $K(\sqrt{d})$. Hence $M_1/K$ is cyclic of degree $2\cdot [M_1 : K(\sqrt{d})]$. Let $M_2$ be the cyclic extension of $K$ with degree $[M_1:K(\sqrt{d})] $ contained in $M_1$. Then $M_2$ is unramifed outside of $S_E$. For $\fp \in S_E$, we have that $\chi_{\cH}(\fp) \neq 0$. Since $[M_2 : K] | \ell$ and $\ell \nmid \cl(K)$, it follows that $M_2$ is not contained in the Hilbert class field of $K$ and is unramified at all primes $K$. Thus, we have that $M_2 = K$, $M_1 = K(\sqrt{d})$ and hence $\widetilde{M} = K(\sqrt{d})$.
\end{proof}

\begin{proof}[Proof of Theorem \ref{1.3}]
The divisibility of $\#\Sel_{\ell}(E^d,K)$ by $\cl_{S_E,u}(K(\sqrt{d}))[\ell]$ follows from Lemmas \ref{4.3}, \ref{4.4} since our element $\phi \in \Sel_{\ell}(E^d,K)$ is induced by $\alpha \in \Gal(M/K(\sqrt{d}))$ and the action of $\langle \delta \rangle$ on $H_{S_E,u}(K(\sqrt{d}))$ does not affect the order of $\alpha$ when considered as an element of $H_{S_E,u}(K(\sqrt{d})).$ 
\end{proof}

\section{Proof of Theorem \ref{C}} \label{S4}
Before we proceed with a proof of Theorem \ref{C}, we wish to shed some light onto our assumptions. In general, our hypotheses allow us to control the ramification in cyclic extensions of $K(\sqrt{d})$.

\begin{remark}[Field assumptions]
We assume that our field $K$ is a number field of degree $n\leq 5$ such that $N_{K/\mQ}(\fq) = 2   \text{ for all }\fq|2$ and that for some $\ell \in S(n)\setminus\brk{2,3} $, $\ell \nmid \cl(K)$ and $\zeta_{\ell} \notin K $. The degree and norm condition appear in Lemma \ref{3.1} and allow us to deduce ramification conditions on prime divisors $\fQ_M|\fq$ where $M_1/K$ is cyclic. The condition that $\ell \nmid \cl(K)$ implies that there does not exist an extension $M_2/K$ of degree $\ell$ contained in the Hilbert class field of $K$; once again this gives us a ramification consequence. The assumption that $\zeta_{\ell} \notin K $ is subtle, but it allows for more ramification possibilities since Kummer theory does not restrict cyclic extensions. The final condition that $e_{\fl}(K/\mQ) \neq 5$ when $[K:\mQ] = 5$ and $\ell = 5$ is due to a deep result of Katz \cite{katz1980galois} concerning the injectivity of $\ell$-torsion under the reduction map; the assumption $1>e_{\fl}(K/\mQ)/(\ell-1) - 1$ from Theorem \ref{D} is the general condition. This assumption allows us to use the fact that prime to 2 torsion will inject under the reduction map.
\end{remark}

To prove Theorem \ref{C}, it suffices to prove the divisibility statement $$\# \Sel_{\ell}(E^d,K) \Big{|} \cl_{\widetilde{S}_E,u}(K(\sqrt{d}))[\ell]\cdot \cl_{S_E}(K')[\ell](\chi_{\ell}).$$
To begin, we discuss the Galois structure of the $\ell$-division field of elliptic curves $E/K$ from Theorem \ref{C}.

\subsection{Galois structure of splitting fields of $\ell$-covers of $E$}\label{4.1}
We want to determine the Galois group structure of splitting fields of elements in $\cH^1(G_K,E(\overline{K})[\ell])$ for elliptic curves having a $K$-rational point $P$ of order $\ell$. Recall that $\zeta_{\ell} \notin K$. Denote the $\ell$-division field by $K(E[\ell])$; this is the field obtained by adjoining the $x,y$ coordinates of all points of order $\ell$ of $E$ to $K$. Then $K(E[\ell])$ is a Galois extension of $K$ containing $K(\zeta_{\ell})$, and it is cyclic over $K(\zeta_{\ell})$ of degree dividing $\ell$. From this point on, we shall abbreviate $E(\overline{K})[\ell]$ with $E[\ell]$, and similarly for $E^d[\ell]$.

\begin{lemma}\label{4.2}
The Galois group $K(E[\ell])/K$ is generated by two elements $\overline{\gamma},\overline{\varepsilon}$ with $\overline{\gamma}^{\ell - 1} = \id$, $\overline{\varepsilon}^{\ell} = \id$, $\overline{ \gamma} | K(\zeta_{\ell})$ generates $K(\zeta_{\ell})/K$, and $\overline{\gamma}\overline{\varepsilon}\overline{\gamma}\inv = \overline{\varepsilon}^{\chi_{\ell}(\overline{\gamma})\inv}$.
\end{lemma}
\begin{proof}
Choose a base of the form $\brk{P,Q}$ of $E[\ell]$ such that for $\sigma \in \Gal(K(E[\ell])/K)$ the action of $\sigma$ on $E[\ell]$ induces the matrix 
$$\rho_{\sigma} = \begin{psmallmatrix} 1 & b \\ 0 & a \end{psmallmatrix} \in \GL_2(\mF_{\ell}),$$
with $a = \det (\rho_{\sigma}) \equiv \chi_{\ell}(\sigma) \mod \ell$. Now we choose $\overline{\gamma}$ such that 
$$\rho_{\overline{\gamma}} = \begin{psmallmatrix} 1 & 0 \\ 0 & w \end{psmallmatrix} \in \GL_2(\mF_{\ell}).$$
with $w$ a generator of $\uZZ{\ell}$. Also, we pick $\overline{\varepsilon} = \id$ if $K(E[\ell]) = K(\zeta_{\ell})$. If $K(E[\ell]) \neq K(\zeta_{\ell})$, we choose $\overline{\varepsilon}$ such that 
$$\rho_{\overline{\varepsilon}} = \begin{psmallmatrix} 1 & 1 \\ 0 & 1 \end{psmallmatrix} \in \GL_2(\mF_{\ell}).$$
Then $\overline{\gamma}$ and $\overline{\varepsilon}$ generate $\Gal(K(E[\ell])/K)$ and since 
$$\begin{psmallmatrix} 1 & 0 \\ 0 & w \end{psmallmatrix}\begin{psmallmatrix} 1 & 1 \\ 0 & 1 \end{psmallmatrix}\begin{psmallmatrix} 1 & 0 \\ 0 & w\inv \end{psmallmatrix} = \begin{psmallmatrix} 1 & w\inv \\ 0 & 1 \end{psmallmatrix} = \begin{psmallmatrix} 1 & 1 \\ 0 & 1 \end{psmallmatrix}^{w\inv}$$
we have the relation $\overline{\gamma}\overline{\varepsilon}\overline{\gamma}\inv = \overline{\varepsilon}^{\chi_{\ell}(\overline{\gamma})\inv}$.
\end{proof}

\begin{remark}
The choice of $\overline{\gamma}$ and $\overline{\varepsilon}$ is closely related to the choice of base $\brk{P,Q}$. In particular, we have $\overline{\varepsilon}(Q) = P + Q$ if $\overline{\varepsilon} \neq \id$ and $\overline{\gamma}(Q) = \chi_{\ell}(\overline{\gamma}) Q$. 
\end{remark}

Let $d \in \cO_K\unit/(\cO_K\unit)^2$ be negative and relatively prime to $\fl\cdot  N(E)$. We define $L_d $ to be the quadratic extension of $K(E[\ell])$ given by the compositum $K(\sqrt{d})\cdot K(E[\ell])$. The Galois group $\Gal(L_d/K)$ is generated by three elements $\delta , \gamma , \varepsilon$ with $\delta$ commuting with $\varepsilon$ and $\gamma$ and
\begin{align*}
\delta^2 &= \id, & &\delta(\sqrt{d}) = -\sqrt{d}, \\
 \gamma^{\ell- 1} &= \id, & & \gamma | K(E[\ell]) = \overline{\gamma}, \\
\varepsilon^{\ell} &= \id, & & \varepsilon | K(E[\ell]) = \overline{\varepsilon}, \\
\gamma^i\varepsilon^j | K(\sqrt{d})& = \id,  & & \gamma\varepsilon\gamma\inv = \varepsilon^{\chi_{\ell}(\gamma)\inv}.
\end{align*}
In particular, we have that $\delta$ operates as $-\id$ on $E^d[\ell]$, the points of order $\ell$ of $E^d$. The fixed field of $\varepsilon$ is $K(\sqrt{d},\zeta_{\ell})$ and the fixed field of $\langle \varepsilon , \delta \gamma^{(\ell - 1)/2} \rangle$ is $K'$ as defined in Theorem \ref{C}. Thus, we have the following field diagram:
$$
\xymatrixrowsep{.2in}
\xymatrixcolsep{.2in}
\xymatrix{
{} & L_d & {} \\
{} & {} & K(E[\ell])\ar@{-}[lu] \\
K'\ar@{-}[ruu] & K(\sqrt{d}) \ar@{-}[uu]   & K(\zeta_{\ell}) \ar@{-}[u]_{\ell} \\
{} & K\ar@{-}[ru]_{\ell-1} \ar@{-}[lu]\ar@{-}[u]^{2}
}
$$

We now describe the elements in $\cH^1(G_K,E^d[\ell])$. We have the exact inflation-restriction sequence
$$0 \lrra \cH^1(\Gal(L_d/K),E^d[\ell]) \overset{\text{inf.}}{\lrra } \cH^1(G_K,E^d[\ell]) \overset{\text{res.}}{\lrra } \cH^1(\Gal(\overline{K}/L_d),E^d[\ell]) ,$$
where $\cH^1(\Gal(\overline{K}/L_d),E^d[\ell]) = \Hom_{\Gal(L_d/K)}(\Gal(\overline{K}/L_d),E^d[\ell]).$

\begin{lemma}
The group $\cH^1(G_K,E^d[\ell])$ injects into $\Hom_{\Gal(L_d/K)}(\Gal(\overline{K}/L_d),E^d[\ell])$.
\end{lemma}
\begin{proof}
We need to show that $\cH^1(\Gal(L_d/K),E^d[\ell]) = 0$. If $\varepsilon = \id$, the degree of $L_d/K$ is prime to $\ell$, and the assertion follows. Now let $\varepsilon$ be of order $\ell$. Using the inflation-restriction sequence, one has that $$\cH^1(\Gal(L_d/K),E^d[\ell]) = \cH^1(\langle \varepsilon \rangle,E^d[\ell])^{\langle \delta, \gamma \rangle}.$$
Let $P_d,Q_d$ be the points of order $\ell$ of $E^d[\ell]$ corresponding to $P,Q \in E[\ell]$ . Then $P_d  = \varepsilon Q_d - Q_d,$
and hence $\cH^1(\langle \varepsilon \rangle , E^d[\ell])$ is generated by the class of cocycle $\psi$ which sends $\varepsilon$ to $Q_d$. Since $\delta\varepsilon\delta = \varepsilon$ and $\delta Q_d = -Q_d$, we have that $\psi \notin \cH^1(\langle \varepsilon \rangle,E^d[\ell])^{\langle \delta \rangle}$, and thus $\cH^1(\Gal(L_d/K),E^d[\ell]) = \cH^1(\langle \varepsilon \rangle,E^d[\ell])^{\langle \delta, \gamma \rangle} = 0$.
\end{proof}

Take an element $\widetilde{\Phi} \in \cH^1(G_K,E^d[\ell])$ with 
$$\res \widetilde{\Phi} = \phi \in \Hom_{\Gal(L_d/K)}(\Gal(\overline{K}/L_d),E^d[\ell])$$
and denote by $M$ the fixed field of the kernel of $\phi$. $M/K$ is normal and $\Gal(M/L_{d})$ is possibly generated by two elements $\alpha_1,\alpha_2$ with $\alpha_i^{\ell} = \id$, which we may choose in such a way that 
$$\phi(\alpha_1) = \mu_1P \quad \aad \quad \phi(\alpha_2) = \mu_2Q .$$
We may also assume that $\mu_i = 1$ if $\alpha_i \neq \id$.

We extend $\delta,\gamma,\varepsilon \in \Gal(L_{d}/K)$  to elements $\widetilde{\delta}, \widetilde{\gamma}, \widetilde{\varepsilon} \in \Gal(M/K)$ and compute that the actions of these elements on $\alpha_i$. We assume that $\widetilde{\delta}^2 = \widetilde{\gamma}^{\ell-1} = \id$. Since 
$$\phi(\beta\alpha_i\beta\inv) = \beta \phi(\alpha_i) \quad \forall \, \beta \in \Gal(M/K)$$
via the fact that $\phi$ is a group homomorphism and the cocycle condition we get:
\begin{align*}
\widetilde{\delta}\alpha_i\widetilde{\delta}&= \alpha_i\inv & &(\because \, \widetilde{\delta}| E^d[\ell] = -\id) \\
\widetilde{\gamma}\alpha_1\widetilde{\gamma}\inv &= \alpha_1 && (\because \,\widetilde{\gamma}P = P), \\
\widetilde{\gamma}\alpha_2\widetilde{\gamma}\inv &= \alpha_2^{\chi_{\ell}(\widetilde{\gamma})} && (\because \,\widetilde{\gamma}Q = \chi_{\ell}(\widetilde{\gamma})Q), \\
\widetilde{\varepsilon}\alpha_1\widetilde{\varepsilon}\inv &= \alpha_1 &&(\because\, \widetilde{\varepsilon}P = P), \\
\widetilde{\varepsilon}\alpha_2\widetilde{\varepsilon}\inv &= \alpha_1\alpha_2 & &\mif \, \varepsilon \neq \id \aad \alpha_2\neq \id (\because \text{ then } \varepsilon\phi(\alpha) = \\
& & & \varepsilon P = P+Q = \phi(\alpha_1\alpha_2); \text{ necessarily } \alpha_1 \neq \id \\
& & & \text{in this case}).
\end{align*}
In particular, it follows that $\langle \alpha_1 \rangle$ is a normal subgroup of $\Gal(M/K)$ and that $\langle \alpha_2 \rangle$ is normal if either $\alpha_2 = \id$ or $\widetilde{\varepsilon} = \id$.

Now we distinguish between two cases:\nl

\hspace*{-1em}\textbf{Case 1.} $\widetilde{\varepsilon} = \id$. In this case $\langle \alpha_1 \rangle$ and $\langle \alpha_2 \rangle$ are both normal in $\Gal(M/K)$ and hence
$$M_i := M^{\langle \alpha_i \rangle}$$
are normal extensions of $K$. The Galois group of $M_2/K(\sqrt{d})$ is abelian and generated by the restriction of $\langle \widetilde{\gamma},\alpha_1 \rangle$ to $M_2$. Hence
$$\overline{M}_2 := M^{\langle \alpha_2,\widetilde{\gamma} \rangle}$$
is Galois over $K$ containing $K(\sqrt{d})$ and if $\alpha_1 \neq \id$, then $\Gal(\overline{M}_2/K)$ is non-abelian of order $2\ell$. Since
$$\widetilde{\delta}\widetilde{\gamma}^{(\ell-1)/2}\alpha_2(\widetilde{\delta} \widetilde{\gamma}^{(\ell-1)/2})\inv = \alpha_2,$$
it follows that $M_1$ is abelian over $K'$ and hence 
$$\overline{M}_1 := M^{\langle \alpha_1,\widetilde{\delta}\widetilde{\gamma}^{(\ell-1)/2}\rangle}$$
is normal over $K$. Its Galois group is generated by 
$$\overline{\alpha}_2 = \alpha_2 | \overline{M}_1 \quad \aad\quad \overline{\gamma} = \widetilde{\gamma} | \overline{M}_1,$$
and its order is equal to $|\alpha_2|\cdot (\ell-1)$. Also one has the relation $\overline{\gamma}\overline{\alpha_2}\overline{\gamma}\inv = \overline{\alpha}_2^{\chi_{\ell}(\overline{\gamma})}.$
To summarize, we have that
\begin{align*}
\overline{M}_1(\phi) &:= M^{\langle \alpha_1,\widetilde{\delta}\widetilde{\gamma}^{(\ell-1)/2}\rangle}, \\
\overline{M}_2(\phi) &:= M^{\langle \alpha_2,\widetilde{\gamma} \rangle}.
\end{align*}

\hspace*{-1em}\textbf{Case 2.} $|\widetilde{\varepsilon} | = \ell$. In this case, we may assume that $\alpha_1 \neq \id$ for $\alpha_1 = \id$ implies that $\alpha_2 = \id$, as well.\vspace*{.5em}

\hspace*{-1em}\textbf{Subcase (i).} $\alpha_2 = \id$. We assert that $\Gal(M/K(\zeta_{\ell},\sqrt{d}))$ is not cyclic. Otherwise $\widetilde{\varepsilon}$ would be an element of order $\ell^2$ with $\widetilde{\varepsilon}^{\,\ell} = \alpha_1$ (without lose of generality). So $\widetilde{\delta}\widetilde{\varepsilon}^{\,\ell}\widetilde{\delta} = \widetilde{\varepsilon}^{-\ell}$ and hence
$$\widetilde{\delta}\widetilde{\varepsilon}\widetilde{\delta} = \widetilde{\varepsilon}^{\,k} \quad \text{with }k\equiv -1\mod {\ell}.$$
But since $\delta\varepsilon\delta = \varepsilon$, we would get 
$\delta\widetilde{\varepsilon}\delta = \widetilde{\varepsilon}\cdot (\widetilde{\varepsilon}^{\ell})^n = \widetilde{\varepsilon}^{\,1+\ell^n}$
which gives a contradiction. Hence, we can choose $\widetilde{\varepsilon}$ so that 
$$\widetilde{\varepsilon}^{\,\ell} = \widetilde{\alpha}_1^{\,\ell} = \id \quad \aad \quad \widetilde{\delta}\widetilde{\varepsilon}\widetilde{\delta} = \widetilde{\varepsilon},$$
which determines $\widetilde{\varepsilon}$ uniquely. Thus, $\overline{M}_2 := M^{\langle \widetilde{\varepsilon} , \widetilde{\gamma}\rangle}$ is normal over $K$ and contains $K(\sqrt{d})$ and its Galois group is dihedral of order $2\ell$ and generated by $\langle \alpha_1,\widetilde{\delta} \rangle$. To summarize, we say that
\begin{align*}
\overline{M}_1(\phi) &:= M^{\langle \alpha_1,\widetilde{\delta}\widetilde{\gamma}^{(\ell-1)/2}\rangle}, \\
\overline{M}_2(\phi) &:= M^{\langle \widetilde{\varepsilon} , \widetilde{\gamma}\rangle}.
\end{align*}

\vspace*{.5em}
\hspace*{-1em}\textbf{Subcase (ii).} $\alpha_2 \neq \id$. We have that $M_1 := M^{\langle \alpha_1\rangle}$ is normal over $K$ and of degree $\ell$ over $L_d$. Since $\widetilde{\delta}\alpha_2\widetilde{\delta} = \alpha_2\inv$, we conclude as above that $\varepsilon$ has an extension $\widetilde{\varepsilon}$ to $M_1$ of order $\ell$ with $\widetilde{\delta}\widetilde{\varepsilon}\widetilde{\delta} = \widetilde{\varepsilon}$. Since $\widetilde{\delta}\widetilde{\gamma}^{(\ell-1)/2}$ acts trivially on $\alpha_2$ and $\widetilde{\varepsilon}$ acts trivially on $\alpha_2 | M_1$, we have that $\langle \widetilde{\delta}\widetilde{\gamma}^{(\ell-1)/2},\widetilde{\varepsilon} \rangle$ is a normal subgroup of $\Gal(M_1/K)$. Hence 
$$\overline{M}_1 := M_1^{\langle \widetilde{\delta}\widetilde{\gamma}^{(\ell-1)/2},\widetilde{\varepsilon} \rangle}$$
is normal over $K$ containing $K'$, and its Galois group over $K'$ is generated by $\overline{\alpha}_2 =\alpha_2 | \overline{M}_1 $, which is of order $\ell$ and satisfies the relation
$$\overline{\gamma}\overline{\alpha}_2\overline{\gamma}\inv = \overline{\alpha}_2^{\chi_{\ell}(\overline{\gamma})} \quad \text{with }\overline{\gamma} = \widetilde{\gamma} | K'.$$
In order to simplify notation, we define $\overline{M}_2(\phi) := K(\sqrt{d})$ if either $\varepsilon \neq \id $ or $\alpha_2 \neq \id$. To summarize, we say that
\begin{align*}
\overline{M}_1(\phi) &:= M_1^{\langle \widetilde{\delta}\widetilde{\gamma}^{(\ell-1)/2},\widetilde{\varepsilon} \rangle}, \\
\overline{M}_2(\phi) &:= K(\sqrt{d}).
\end{align*}

Hence for a given 
$$\widetilde{\Phi} \in \cH^1(G_K,E^d[\ell])$$
we have a field $M = M(\phi)$ which determines $\langle \phi \rangle$ completely where $\phi = \res(\widetilde{\Phi})$. We want to study the information we attain from the pair $(\overline{M}_1(\phi),\overline{M}_2(\phi))$. If $\varepsilon = \id$ or $\alpha_2 = \id$, then we get back $M(\phi)=M$ from $(\overline{M}_1(\phi),\overline{M}_2(\phi))$. In these cases, we shall say that $\phi$ is of \cdef{first type}. What happens if $\varepsilon \neq \id$ and $\alpha_2 \neq \id$? Assume that 
$$\phi \neq \psi \in \cH^1(G_K,E^d[\ell])$$
have fields $M(\phi)$ and $M(\psi)$ with Galois groups $\langle \alpha_1,\alpha_2 \rangle$ and $\langle \beta_1 , \beta_2\rangle$ as above such that 
$$M(\phi)^{\alpha_1} = M(\psi)^{\beta_1}. $$
Let $N$ be the composite of $M(\phi)$ and $M(\psi)$. Then the Galois group $\Gal(N/L_d)$ is generated by three elements $\langle \sigma_1,\sigma_2,\sigma_3 \rangle$, which we can choose in such a way that
$$\sigma_1 | M(\phi) = \alpha_1 ,\gap  \sigma_1 |M(\psi) = \beta_1^{\lambda}$$
$$\sigma_2 | M(\phi) = \alpha_2 ,\gap  \sigma_2 |M(\psi) = \beta_2^{\lambda}$$
where $\lambda \in \brk{1,\dots , \ell -1}$. $N$ is a splitting field of $\phi$ and $\psi$, and
$$(\phi - \lambda^{\ell - 1}\psi)(\sigma_1) = (\phi - \lambda^{\ell - 1}\psi)(\sigma_2) = 0.$$
Hence the fixed field of the kernel of $\phi - \lambda\psi$ is a cyclic extension of $L_d$ which is normal over $K$, and $\phi - \lambda\inv\psi$ is of first type. 

Thus, $\overline{M}_1(\phi)$ determines $\langle \phi \rangle$ up to elements of first type, and in order to determine all elements in $\cH^1(G_K,E^d[\ell])$, it is enough to determine all dihedral extensions of $K$ of degree $2\ell$ containing $K(\sqrt{d})$ and all extensions $M_1$ of degree $\ell$ over $K'$ which are normal over $K$ such that conjugation by $\overline{\gamma}$ on $\Gal(\overline{M}_1,K')$ is equal to $\chi_{\ell}(\overline{\gamma})$. 

Therefore to prove the double divisibility, one has to show that for $\phi \in \Sel_{\ell}(E^d,K)$, the field $\overline{M}_2(\phi)$ is unramified over $K(\sqrt{d})$ outside $\widetilde{S}_E $, and $\overline{M}_1(\phi)$ is unramifed over $K'$ outside $S_E$ and little ramified at divisors of $\fl$.

\subsection{Splitting fields of elements in $\Sel_{\ell}(E^d,K)$}\label{S5}
We shall continue to use the assumptions and the notations of the Theorem \ref{C} and Section \ref{4.1}.
\begin{lemma}\label{3.1}
Let $\phi$ be an element in $\Sel_{\ell}(E^d,K)$. Then $\overline{M}_1(\phi) =: \overline{M}_1$ is unramified at $\fq$ over $K'$ and $\overline{M}_2(\phi) =: \overline{M}_2$ is unramified at $\fq$ over $K(\sqrt{d})$.
\end{lemma}

\begin{proof}
We first prove the latter statement. Since $\fq | \Delta_{K(\sqrt{d})/K}\,$, we have that $K(\sqrt{d})$ and $K'$ are ramified at $\fq$ over $K$. Hence the norm of $\fQ|\fq$ in $K(\sqrt{d})$ is equal to $\fq$, and by assumption the norm of $\fQ|2$ is equal to 2. Suppose that $K(\sqrt{d})$ had a cyclic extension of degree $\ell$ in which $\fQ$ is ramified. Then the completion $K(\sqrt{d})_{\fQ}$ admits a cyclic extension of degree $\ell$ ramified at $\fQ$. Since $\ell$ is odd and $\fQ$ has residue characteristic two, this extension is tamely ramified. By local class field theory, the tamely ramified cyclic extensions of a local field $K(\sqrt{d})_{\fQ}$ all have degree dividing $|\kappa\unit|$, where $\kappa$ is the residue field. Since $\kappa = \mF_2$, we have that there are no tamely ramified and ramified extensions of $K(\sqrt{d})_{\fQ}$. Thus, $K(\sqrt{d})$ has no cyclic extension of degree $\ell$ in which $\fQ$ ramifies, and hence $\overline{M}_2$ is unramified at $\fq$ over $K(\sqrt{d})$.

To prove the former statement, we shall utilize the proof of \cite[Lemma~3]{frey1987selmer} and look prime by prime. For $\ell = 5$, the same argument as above can be applied to $\fQ_{K'}|\fq$. For $\ell = 7$, there is only one extension $\fQ|\fq$ to $K'$ which is ramified of order $2$ and has norm $8$. Assume that $\fQ_{K'}$ is ramified in $\overline{M}_1/K'$ and let $\fQ_{\overline{M}_1}$ be the unique extension of $\fQ_{K'}$ to $\overline{M}_1$. Let $M_t$ be the subfield of $\overline{M}_1$ in which $\fQ_{\overline{M}_1}$ is tamely ramified. Then $M_t$ is a cyclic extension of degree $7$ over $K(\zeta_{7} + \zeta_{7}\inv)$, and $\overline{M}_1$ is the compositum of $M_t$ with $K'$ over $K(\zeta_{7} + \zeta_{7}\inv)$. Thus, $\Gal(\overline{M}_1/K(\zeta_{7} + \zeta_{7}\inv))$ is abelian. But this contradicts the fact that 
$$\overline{\gamma}^{3}\overline{\alpha}\overline{\gamma}^{3} = \overline{\alpha}^{\chi_{7}(\overline{\gamma}^{3})} = \overline{\alpha}\inv,$$
where $\langle \overline{\alpha}\rangle = \Gal(\overline{M}_1/K')$ and $\langle \overline{\gamma}\rangle = \Gal(K'/K)$. 

For $\ell = 11,13,19,37$, we can use the same proof as the first statement since 
\begin{align*}
11 &\nmid (2^5 - 1)   & 13 &\nmid  (2^2 - 1) & 37 &\nmid (2^{3} - 1) & 37 &\nmid (2^{18} - 1) \\
13 &\nmid (2^6 - 1)   & 19 &\nmid (2^9 - 1)   & 37 &\nmid (2^{9} - 1) &  37 &\nmid (2^{2} - 1). \\
13 &\nmid  (2^3 - 1)  & 19 &\nmid (2^3 - 1)  & 37 &\nmid (2^{6} - 1) &  & 
\end{align*}
For $\ell = 17$, there there is only one extension $\fQ|\fq$ to $K'$ which is ramified of order $2$ and has norm $2^8$ (note that $17|(2^8 - 1)$). If we assume that $\fQ_{K'}$ is ramified in $\overline{M}_1/K'$, then we can use the above argument to construct the same contradiction.
\end{proof}

\begin{remark}
Since $73 | (2^{36} - 1),$ $73 | (2^9 - 1)$, and $73 | (2^{18} - 1)$, we may not assume that there is a unique cyclic extension of $K'$ with degree $73$ in which $\fQ$ is ramified, and hence the above argument does work for $\ell = 73$. This precludes us from extending Theorem \ref{C} to number fields $K$ of degree $6$.
\end{remark}

Therefore, we can assume that $\fp \nmid \fq \cdot \fl$, but $\fp | N(E)$.
\begin{lemma}\label{3}
Let $\phi$ be an element in $\Sel_{\ell}(E^d,K)$. Then $ \overline{M}_1/K'$ is unramified outside of $S_E \cup \brk{\fl}$ and $ \overline{M}_2/K(\sqrt{d})$ is unramified outside $\widetilde{S}_E \cup \brk{\fl} $.
\end{lemma}

\begin{proof}
We have to test prime numbers $\fp\neq \fl$ that divide $N(E)$.
\begin{enumerate}
\item If $\ord_{\fp}(j_E) \geq 0$, then it follows from N\'eron's list of minimal models of elliptic curves with potentially good reduction that $\ell$ must be equal to 3 (\cite[p.124]{NeronMinModel}). Since we only consider primes $\ell > 3$, we can exclude this case from consideration.\\
\item Now assume that $\ord_{\fp}(j_E) < 0$. We have two subcases:\vspace*{.5em}

\hspace*{-1em}\text{(a)} If $\ord_{\fp}(j_E) \equiv 0 \mod \ell$, we have that $\fp \notin S_E$ and so $E^d$ is not a Tate curve over $K_{\fp}.$ Moreover, $K_{\fp}(E[\ell])$ is unramified over $K_{\fp}$ and hence $\overline{M_1}/K'$ and $\overline{M_2}/K(\sqrt{d})$ are unramified at all divisors of $\fp$ if and only if $M_1/L_d$ (resp. $M_2/L_d$) are unramified at all divisors of $\fp$. We now use the triviality of the $\phi \in \Sel_{\ell}(E^d,K)$ over $K_{\fp}$ from Lemma \ref{4.3}. Also recall that $M$ is the fixed field of the kernel of $\phi$. We shall show that $\fQ_M$ is unramified over $L_d$.

There is a $\widetilde{P} \in E^d(M_{\fP})$ where $\fP_M|\fp$ such that for all $\sigma \in D(\fP_M)$, we have $\sigma\widetilde{P}  - \widetilde{P} = \phi(\sigma)$. Hence 
$$P' := \ell \cdot \widetilde{P} \in E^d(K_{\fp})$$
and so $2 P'$ is in the connected component of unity modulo $\fp$ via Remark \ref{2.19}. Hence $\widetilde{P} = \widetilde{P}_1 + P_2 $ with $P_2 \in E^d[\ell]$ and $2\widetilde{P}_1$ in the component of the unity of $E \mod \fP_M$, so $\widetilde{P}_1$ corresponds to a $\fP_M$-adic unity $u$ under the Tate parametrization. Now take 
$$\alpha \in \langle \alpha_1,\alpha_2 \rangle \cap I(\fP_M)$$
where $I(\fP_M)$ is the interia group of $\fP_M$. Then $2(\alpha\widetilde{P} - \widetilde{P})$ corresponds to $\alpha u / u$ and is an $\ell\tth$ root of unity. By Hilbert's Theorem 90, we have that $\alpha = \id$, and thus, $\fP_M$ is unramified over $L_d$.\vspace*{.5em}

\hspace*{-1em}\text{(b)} If $\ord_{\fp}(j_E) \not\equiv 0\mod \ell$, then the values at the Hecke characters $\chi$ of order $\ell$ tell us that either $E$ is a Tate curve over $K_{\fp}$ or that $\fp\in S_E$. Consider the former situation. Our assumptions from Theorem \ref{C} tell us that $\fq$ is not completely decomposed in $K(\sqrt{d})$ and $K'$. Since 
$$K_{\fp}\unit/(K_{\fp}\unit)^{\ell}\iso K_{\fp}(\sqrt{d})\unit/(K_{\fp}(\sqrt{d})\unit)^{\ell} \iso K'^{\times}_{\fP}/(K'^{\times}_{\fP})^{\ell}$$
for all $\fP_{K'}|\fp$, we see that for all cyclic extensions $\overline{M}_1$ of $K'$ and $\overline{M}_2/K(\sqrt{d})$ of degree $\ell$ and divisors $\fP_{M_i}|\fp$, one has that $\Gal(\overline{M}_{i,{\fP_{M_i}}}/K_{\fq})$ is abelian of even order. But this implies that 
$$\overline{M}_{1,\fP} = K_{\fp}' \quad \aad \quad \overline{M}_{2,\fP} = K_{\fP}(\sqrt{d}), $$
which is absurd. Thus $\fp \in S_E$ and our lemma follows. 
\end{enumerate}
\end{proof}
The next step is to describe the behavior of $\overline{M}_i$ at divisors of $\fl$.

\begin{lemma}\label{3.2}
Assume that $\ord_{\fl}(j_E) < 0$ and $\phi \in \Sel_{\ell}(E^d,K)$. Then $\overline{M}_2/K(\sqrt{d})$ is unramified at $\fl$ and $\overline{M}_1/K'$ is little ramified at divisors of $\fl$.
\end{lemma}

\begin{proof}
The assumptions tells us that $E/K_{\fl}$ is a Tate curve but that $E^d/K_{\fl}$ is not a Tate curve. Since $K_{\fl}(E[\ell]) = K_{\fl}(\zeta_{\ell})$, the behavior of $\overline{M}_i$ at $\fl$ is determined by the behavior of $M$ at $\fl$. Let $\fL_M|\fl$, let $I(\fL_M) $ be the inertia group of $\fL_M$, and let
$$\alpha \in \langle \alpha_1,\alpha_2 \rangle \cap I(\fL_M).$$
As in the proof of Lemma \ref{3}, we can use the fact that $E^d/K_{\fl}$ is not a Tate curve to show that there is a $\widetilde{Q} \in E^d(M_{\fL})$ where $\fL_M|\fl$ and $\alpha\widetilde{Q}  - \widetilde{Q} = \phi(\alpha)$. Hence $2 \widetilde{Q}$ is in the connected component of unity modulo $\fL_M$ via Remark \ref{2.19}. This implies that 
$$M_{\fL_M} = M_{\fL_M}^{\langle\alpha\rangle}(\sqrt[\ell]{u})$$
where $u$ is a $\fL_M$-adic unit corresponding to $2\widetilde{Q}$ under the Tate parametrization. Moreover, $M_1/L_d$ is little ramified at $\fl$.

Now assume that $\alpha_2 = \id$ or $\varepsilon = \id$. Then $\overline{M}_2/K(\sqrt{d})$ is of degree $\ell$, and we have to show that $\overline{M}_2/K(\sqrt{d})$ is unramified at $\fL_{\overline{M}_2}|\fl$. We recall the choice of point $Q$. Since $\gamma Q = \chi_{\ell}(\gamma)Q$ where $\langle \gamma \rangle = \Gal(K(\zeta_{\ell})/K)$, it follows that $Q$ is in the kernel of the reduction of $E$ modulo all divisors of $\fl$, and hence $P + \lambda Q$ is not in this kernel where $\lambda \in \mN$. For $\alpha \in I(\fL_M)$, we saw that $\sigma\widetilde{Q}  - \widetilde{Q} = \phi(\sigma)$ is in the kernel of the reduction modulo $\fL_M$, and hence
$$\alpha_1\alpha_2^{\lambda} \notin I(\fL_M) \quad \forall \, \lambda \in \mN \aad \fL_M | \fl.$$
Thus, it follows that $M^{\langle \alpha_2 \rangle}/L_d$ is unramified at $\fL_M$ and $\overline{M}_2/K(\sqrt{d})$ is unramified at $\fl$. 
\end{proof}

Finally, we look at the case where $\ord_{\fl}(j_E) \geq 0$.
\begin{lemma}\label{3.3}
Assume that $E/K$ has a $K$-rational point $P$ of order $\ell > 3$, that $\ord_{\fl}(j_E) \geq 0$, and that $P$ is not contained in the kernel of reduction modulo $\fl$, in particular, this means that $E$ is not supersingular modulo $\fl$. Let $\phi$ be an element in $\Sel_{\ell}(E^d,K)$ with corresponding fields $\overline{M}_1$ and $\overline{M}_2$. Then $\overline{M}_1/K'$ is little ramified at $\fl$, and $\overline{M}_2/K(\sqrt{d})$ is unramified at $\fl$.
\end{lemma}

\begin{proof}
Suppose that $\ord_{\fl}(j_E) \geq 0$, which implies that $E$ has potentially good reduction at $\fl$. Since $E/K$ has a $K$-rational point $P$ of order $\ell>3$, we know that $\Gal(K(E[\ell])/K(\zeta_{\ell}))$ is a subgroup of the additive group $\mF_{\ell}^+$. We want to show that all divisors of $\fl$ are not ramified in $K(E[\ell])/K(\zeta_{\ell})$. If $E$ has good reduction over $K(\zeta_{\ell})$, then we are immediately done. If $E$ does not have good reduction over $K(\zeta_{\ell})$, 
then there must exist some extension $N/K(\zeta_{\ell})$ such that $[N:K(\zeta_{\ell})] | 6$ and that $E$ has good reduction at all divisors $\fL_N | \fl$; this divisibility condition is similar to the proof of \cite[Proposition~VII.5.4.c]{silvermanAEC}. From our assumptions, it follows that $N_{\fL}$ contains $K(E[\ell])$ and that $\langle Q \rangle$ is the subgroup of order $\ell$ of the kernel of reduction modulo $\fL_N$. Hence all divisors of $\fl$ are not ramified in $K(E[\ell])/K(\zeta_{\ell})$, and we can prove the lemma by looking at the behavior of $\fl$ in $M/L_d$.

Assume that $\fL_M|\fl$ and let $I(\fL_M)$ be the inertia group of $\fL_M$. Suppose that $\alpha_1^{\mu}\alpha_2^{\lambda}\in I(\fL_M).$ There there is a $\widetilde{P} \in E(M_{\fL})$ with
$$(\alpha_1^{\mu}\alpha_2^{\lambda})\widetilde{P} - \widetilde{P} = \mu P + \lambda Q.$$
But we know that for $\mu \neq 0$, the point $\mu P  + \lambda Q$ is not in the kernel of reduction modulo $\fL_M$. Let $\widetilde{E}$ be a model of $E$ over $N$ having good reduction modulo $\fL_N|\fl$. Since $(I(\fL_M) - \id)\widetilde{E}(N\cdot M_{\fL})$ is contained in this kernel, we must have that $\mu = 0$, and hence
$$I(\fL_M) \cap \Gal(M/L_d) \subseteq \langle \alpha_2 \rangle.$$
Thus, $M^{\langle \alpha_2 \rangle}/L_d$ is unramified at $\fL_M$; moreover, $\overline{M}_2/K(\sqrt{d})$ is unramified above $\fl$.

Now assume that $I(\fL_M) = \langle \alpha_2 \rangle$. Then $Q = \alpha_2\widetilde{Q} - \widetilde{Q}$  and since $\langle \alpha_2 \rangle$ acts trivially on $\widetilde{E}(N\cdot M_{\fL})/\widetilde{E}_-(N\cdot M_{\fL})$, we may assume that $\widetilde{Q} \in \widetilde{E}_-(N\cdot M_{\fL})$ and hence
$\ell\cdot \widetilde{Q} \in \widetilde{E}_-(N\cdot K_{\fl}).$ Since $\widetilde{E}$ has ordinary reduction modulo $\fL_M$, we have that $N\cdot K_{\fl}(\widetilde{Q})$ is little ramified at divisors of $\fl$. Thus, our lemma follows.
\end{proof}
Lemmas \ref{3.1}, \ref{3}, \ref{3.2}, \ref{3.3} prove that for $\phi \in \Sel_{\ell}(E^d,K)$, the field $\overline{M}_2(\phi)$ is unramified over $K(\sqrt{d})$ outside $\widetilde{S}_E \cup \brk{\fl}$, and $\overline{M}_1(\phi)$ is unramifed over $K'$ outside $S_E$ and little ramified at divisors of $\fl$. Moreover, we have proved that $$\# \Sel_{\ell}(E^d,K) \Big{|} \cl_{\widetilde{S}_E,u}(K(\sqrt{d}))[\ell]\cdot \cl_{S_E}(K')[\ell](\chi_{\ell}),$$
which completes the proof of Theorem \ref{C}.

\subsection*{Proof of Corollary \ref{E}}
Since we have established our double divisibility statement \eqref{eqn1}, we can proceed with a proof of Corollary \ref{E}. By the definitions established in Section \ref{S2}, we have that 
$$ \cl_{\widetilde{S}_E,u}(K(\sqrt{d}))[\ell]\cdot \cl_{S_E}(K')[\ell](\chi_{\ell})   \Big{|} \cl_{\emptyset,u}(K(\sqrt{d}))[\ell]\cdot \cl_{\emptyset}(K')[\ell](\chi_{\ell})\cdot \varepsilon_S$$
where $\varepsilon_S$ is a number depending only on $\widetilde{S}_E$. Note that when $\widetilde{S}_E = \emptyset$, we have that $\varepsilon_S = 1$ and that $\cl_{\emptyset,u}(K(\sqrt{d}))[\ell] = \cl(K(\sqrt{d}))[\ell]$ by Remark \ref{2.6}.  Corollary \ref{E} follows immediately from the following lemma.
\begin{lemma}
$\cl_{\emptyset}(K')[\ell](\chi_{\ell}) \, | \, \cl(K(\sqrt{d}))[\ell]$. 
\end{lemma}
\begin{proof}
Let $M/K$ be a Galois extension containing $K'$ with $\langle \alpha \rangle = \Gal(M/K)$, with the relations
$$\alpha^{\ell} = \id \quad \aad \quad \overline{\gamma}\overline{\alpha}\overline{\gamma}\inv = \alpha^{\chi_{\ell}(\overline{\gamma})} \, \mwhere\, \langle \overline{\gamma} \rangle = \Gal(K'/K).$$
We assume that $M$ is unramified outside $\fl$ and little ramified at $\fl$; hence 
$$M(\zeta_{\ell}) = K'(\sqrt{d})(\sqrt[\ell]{c}),$$
with $c \in M(\sqrt{d})$ and the principal divisor of $c$ is a $\ell\tth$ power. We want to extend $c$ to an element of order $\ell$ in the divisor class group of $K(\sqrt{d})$.

Let $\widetilde{\gamma}$ be an extension of $\overline{\gamma}$ to $\Gal(M(\sqrt{d})/K)$ such that $\widetilde{\gamma}^{\ell - 1} = \id$, $\widetilde{\gamma} |{K(\zeta_{\ell})}$ generates $\Gal(K(\zeta_{\ell})/K)$, and $\widetilde{\gamma} |{K(\sqrt{d})} = \id$. Since $M(\sqrt{d}) / K$ is normal, we have $\widetilde{\gamma}(c) = c^i\cdot e^{\ell}$ with $1\leq i \leq \ell -1$ and $e \in K'(\sqrt{d})$. Hence, 
$$\widetilde{\gamma}(\sqrt[\ell]{c}) = (\sqrt[\ell]{c})^i\cdot e \cdot \xi_{\widetilde{\gamma}}$$
with $\xi_{\widetilde{\gamma}}^{\ell} = 1$. Let $\widetilde{\alpha}$ be an extension of $\alpha$ to $M(\sqrt{d})$ of order $\ell$. We can see that $i = 1$ since
$$\widetilde{\gamma}\widetilde{\alpha}(\sqrt[\ell]{c}) = \xi_{\widetilde{\alpha}}^{\chi_{\ell}(\overline{\gamma})}\widetilde{\gamma}(\sqrt[\ell]{c})$$
and 
$$
\widetilde{\alpha}^{\chi_{\ell}(\gamma)}\widetilde{\gamma}(\sqrt[\ell]{c}) = \widetilde{\alpha}^{\chi_{\ell}(\gamma)}(\xi_{\widetilde{\gamma}}(\sqrt[\ell]{c})^i \cdot e) = \xi_{\widetilde{\alpha}}^{i\cdot \chi_{\ell}(\overline{\gamma})}\cdot \widetilde{\gamma}(\sqrt[\ell]{c}),
$$
and hence
$$M(\sqrt{d}) = K(\sqrt{d}, \sqrt[\ell]{c},\zeta_{\ell}).$$
There exists an element $\widetilde{c}  =  c^{\ell - 1}\cdot e'^{\ell} \in M(\sqrt{d})$ with $e' \in K'(\sqrt{d})$ such that
the divisor of $\widetilde{c}$ is a $\ell\tth$ power. However, since $\pm \widetilde{c}$ is not an $\ell\tth$ power in $K(\sqrt{d})$, it is an element of order $\ell$ in the divisor class group of $K(\sqrt{d})$.
\end{proof}

\section{Elliptic curves satisfying Corollary \ref{D}}\label{S6}
Let $E$ be an elliptic curve over a number field $K$. In a recent work \cite{zywinapossible}, Zywina has described all known, and conjecturally all, pairs $(E/\mQ,\ell)$ such that mod $\ell$ image of Galois, $\rho_{E,\ell}(G_{\mQ})$, is non-surjective. Using Zywina's classification, we can find elliptic curves $E/\mQ$ that will satisfy the conditions of Corollary \ref{D}. First, we present an example of this technique for the case when $\ell = 3$. We remark that this case does not apply to Corollary \ref{D}; however, it best illustrates the technique.

Let $E/\mQ$ be a non-CM elliptic curve over $\mQ$ such that $\rho_{E,3}(G_{\mQ})$ conjugate to $$
B(3):= \begin{psmallmatrix} * & * \\ 0 & *\end{psmallmatrix} \subset \GL_2(\mF_{3}) .
$$ We can use Galois theory to prove the following result:
\begin{prop}\label{33.1}
Let $E/\mQ$ have mod 3 image of Galois conjugate to $B(3)$. Then $\mQ(E[3]) = \mQ(x(E[3])) \cdot K$ where $K$ is an explicitly computable quadratic extension.
\end{prop}
Before we prove Proposition \ref{33.1}, we prove the following lemma which tells us over which extension $E$ obtains a $3$-torsion point.
\begin{lemma}\label{33.2}
For $E/\mQ$ from Proposition \ref{33.1}, there exists some quadratic extension $K$ such that $E$ has a $K$-rational $3$-torsion point. In particular, $E(K)[3] = \langle P \rangle$. 
\end{lemma}
\begin{proof}
Let $E\colon y^2 = x^3 - Ax - B$ for $A,B\in \mQ$. Via the Weil-pairing, we know that $\mQ(\zeta_3) \subseteq \mQ(E[3])$. It is also a well known fact that $B(3) \iso S_3 \times \ZZ{2}$. Combining these results with our assumptions, we have the following diagram of Galois sub-fields of $\mQ(E[3])$:
\begin{figure*}[h!]
$$
\xymatrix{
\mQ(E[3])  \\
\mQ(x(E[3])) \ar@{-}[u]^{2}\\
\mQ(\zeta_3)\ar@{-}[u]^{3}\\
\mQ \ar@{-}[u]^{2} 
}
$$
\end{figure*}
where the extension $\mQ(x(E[3]))$ is the index $2$ sub-field of $\mQ(E[3])$ generated by the $x$-coordinates of points in $E(\overline{\mQ})[3]$. Recall that the roots of the $3$-division polynomial 
$$\psi_3(x) = 3x^4 + 6Ax^2 + 12Bx -A^2$$
correspond to $x$-coordinates of $E(\overline{\mQ})[3]$. In particular, $\psi_3(x)$ is the minimal polynomial of the degree $6$, Galois extension $\mQ(x(E[3]))$. 

Since $S_4$ does not contain any transitive subgroups of order 6, we know that $\psi_3(x)$ must have a linear factor, so we write $\psi_3(x)  = (x-\alpha)g(x)  $ where $\alpha \in \mQ $ and $g(x)$ is an irreducible cubic. This implies that there exists some $P\in E(\overline{\mQ})[3]$ with $\mQ$-rational $x$-coordinate given by $\alpha$. Moreover, we see that there is a 3-torsion point 
$$P = (\alpha , \sqrt{f(\alpha)}).$$
that is defined over the quadratic extension $\mQ(\sqrt{f(\alpha)})$.
\end{proof}
\begin{remark}\label{6.3}
From the above proof, one can easily see that $\Gal(\mQ(x(E[3]))/\mQ) \iso S_3$. Indeed, since $\mQ(x(E[3]))$ is Galois, we showed that the Galois group of $\psi_3(x)$ is actually the Galois group of the cubic $g(x)$. Since $[\mQ(x(E[3])) : \mQ] = 6$, we know $g(x)$ must be an irreducible cubic with non-square discriminant, which immediately implies our claim. 
\end{remark}
\begin{proof}[Proof of Proposition \ref{33.1}]
Let $K$ denote the quadratic extension from Lemma \ref{33.2}. It is clear that $K \subset \mQ(E[3])$ and that $K \nsubseteq \mQ(x(E[3]))$, so we have $\mQ(E(3))$ is the compositum of $\mQ(x(E[3]))$ and $K$.
\end{proof}

The idea behind finding elliptic curves over $\mQ$ such that $E(\mQ)[\ell] = \brk{\cO}$ and $E(K)[\ell] = \langle P \rangle$ is to consider $E/\mQ$ with $\rho_{E,\ell}(G_{\mQ})$ conjugate to a subgroup $H$ such that
$$  
\begin{psmallmatrix} 1 & * \\ 0 & *\end{psmallmatrix} \subsetneq H \subseteq 
\begin{psmallmatrix} * & * \\ 0 & *\end{psmallmatrix} =: B(\ell) .
$$
We can see that $E$ will attain an $\ell$ torsion point over an extension $K$ where the degree of $K/\mQ$ is determined the cardinality of the upper left entry. For $\ell = 3$, we saw that $H = B(3)$ and thus the upper left entry has order 2, which gives a less explicit proof of Proposition \ref{33.1}.

Let $\ell \in \brk{5,13}$. Below, we provide examples of elliptic curves $E/\mQ$ that do not have a $\mQ$-rational point of order $\ell$ but attain a $K$-rational point $P$ of order $\ell$ over some extension of small degree $K$ that satisfies the conditions of Corollary \ref{D}. The final step in our verification is showing $P$ is not contained in the kernel of reduction modulo $\fl$; in particular, this means that $E/K$ is not supersingular modulo $\fl$ if $\ord_{\fl}(j_E) \geq 0$. This condition is computable via the \textsc{Magma} command \texttt{IsSupersingular}. 

In order to conduct a thorough search, we consider all subgroups $H$ which can occur as an image of Galois for a non-CM $E/\mQ$ and satisfy the above containment. In particular, we run through a large list elliptic curves $E/\mQ$ with prescribed non-surjective mod $\ell$ image of Galois coming from the modular curves $X_H$ of Zywina \cite{zywinapossible}. Since this list is comprehensive, we also give examples of elliptic curves over $\mQ$ that do not satisfy and potentially satisfy Corollary \ref{D}, modulo some computations.

For $\ell = 5$, we only have one example.

\begin{exam}[$\ell = 5$]\label{6.3.1}
Let $E/\mQ$ be the elliptic curve
$$E\colon y^2 = f(x) = x^3 - \frac{185193}{185193}x + \frac{185193}{149}.$$
$E $ has mod 5 image of Galois conjugate to $B(5) \subset \GL_2(\mF_5)$, and hence $E$ attains a $K$-rational point of order $5$ over a bi-quadratic extension $K$ of $\mQ$. The first quadratic extension $L/\mQ$ is given by adjoining the quadratic root $\alpha$ of the $5$-division polynomial $\psi_5$, and then the second quadratic is given by adjoining the square root of the $f(\alpha)$. For $E$ defined above, we compute that $\cl(K) = 8$, $\zeta_5\notin K$, $2$ is ramified in $\cO_K$, and that $E/K$ is not supersingular modulo $\fl$ if $\ord_{\fl}(j_E) \geq 0$ where $\fl | 5$. Therefore, the elliptic curve $E$ and the number field $K$ satisfy the conditions of Corollary \ref{D}.
\end{exam}

For $\ell = 7$, we have two possibilities.

\begin{almexam}[$\ell = 7$]\label{6.3.2}
Let $E/\mQ$ be the elliptic curve
$$E\colon y^2 = f(x) = x^3 - \frac{81469949623875}{3017401762489}x + \frac{162939899247750}{3017401762489},$$
which has mod $7$ image conjugate to $B(7)$. $E$ attains a $K$-rational point of order $7$ over an extension $K$ of degree 6. The extension $K$ is given by first adjoining the root $\alpha$ of the cubic factor of $\psi_7$ and then adjoining the square root of $f(\alpha)$. We verify almost all of the conditions from Corollary \ref{D} for $E$ and $K$; however, we are not able to verify that $7 \nmid \cl(K) $.
\end{almexam}

\begin{nonexam}[$\ell = 7$]\label{6.3.3}
Suppose that $E/\mQ$ has $\rho_{E,7}(G_{\mQ})$ conjugate to
$$
H:= \begin{psmallmatrix} a^2 & * \\ 0 & *\end{psmallmatrix} \quad \mwhere a\in \mF_7 .
$$
Since $\#(\mF_7\unit)^2 = 3$, we have that $E$ attains a $K$-rational point of order $7$ over a cubic extension $K$. Moreover, this extension is given adjoining the root of the cubic factor of the $7$-division polynomial $\psi_7$. In our search, we find that all $E/K$ are supersingular modulo $\fl$ if $\ord_{\fl}(j_E) \geq 0$ where $\fl | 7$.

\end{nonexam}

For $\ell = 11$, there do not exist any subgroups coming from \cite{zywinapossible} that have our desired condition. For $\ell = 13$, we find a few examples of curves satisfying Corollary \ref{D}.

\begin{exam}[$\ell = 13$]
Suppose that $E/\mQ$ has $\rho_{E,13}(G_{\mQ})$ conjugate to 
$$
H =\begin{psmallmatrix} a^3 & * \\ 0 & *\end{psmallmatrix} \quad \mwhere a\in \mF_{13},
$$
then $E$ attains a $K$-rational point of order 13 over a bi-quadratic extension $K/\mQ$ since $\#(\mF_{13}\unit)^3 = 4$. As an example, consider the elliptic curve
$$E\colon y^2 = x^3 - \frac{2248091}{180353}x + \frac{4496182}{180353},$$
which has mod 13 image conjugate to $H$. $E$ attains a $K$-rational point of order $13$ over a bi-quadratic extension $K$ of $\mQ$. The first quadratic extension $L/\mQ$ is given by adjoining a quadratic root $\alpha$ of the $13$-division polynomial $\psi_{13}$, and then the second quadratic is given by adjoining the square root of the $f(\alpha)$. We compute that $\cl(K) = 2$, $\zeta_{13}\notin K$, $(2)$ splits in $\cO_K$, and $E/K$ is not supersingular modulo $\fl$ if $\ord_{\fl}(j_E) \geq 0$ where $\fl | 13$. Therefore, the elliptic curve $E$ and the number field $K$ satisfy the conditions of Corollary \ref{D}.
\end{exam}

\begin{exam}[$\ell = 13$]\label{6.3.6}
Suppose that $E/\mQ$ has $\rho_{E,13}(G_{\mQ})$ conjugate to 
$$
H := \begin{psmallmatrix} a^4 & * \\ 0 & *\end{psmallmatrix} \quad \mwhere a\in \mF_{13} .
$$
Since $\#(\mF_{13}\unit)^4 = 3$, $E$ attains a $K$-rational point of order 13 over cubic extension $K/\mQ$. For example, consider the elliptic curve $$E\colon y^2 = x^3 + 13674069x + 324405221670.$$
Using \cite{zywinapossible}, $E$ has mod 13 image conjugate to $H$. Now let $K/\mQ$ denote the number field defined by the cubic factor of $\psi_{13}$. For notational purposes, we shall write $K =\mQ(\alpha)$ where $\alpha$ is the primitive element of $K$. By base changing to $K$, we find that $E_K = E\times_{\mQ} K$ has $K$-rational $13$-torsion point. We also compute that $\cl (F) = 1$, $2$ splits in $\cO_K$, $\zeta_{13} \notin K$, and that $E/K$ is not supersingular modulo $\fl$ if $\ord_{\fl}(j_E) \geq 0$ where $\fl | 13$. Therefore, the elliptic curve $E$ and the number field $K$ satisfy the conditions of Corollary \ref{D}.
\end{exam}

\begin{exam}[$\ell = 13$]
Suppose that elliptic curve with $\rho_{E,13}(G_{\mQ})$ conjugate to
$$
\begin{psmallmatrix} a^2 & * \\ 0 & *\end{psmallmatrix} \quad \mwhere a\in \mF_{13} .
$$
Since $\#(\mF_{13}\unit)^2 = 6$, $E$ will attain a $K$-rational point of order $13$ over an extension of degree 6. As an example, consider the elliptic curve
$$E\colon y^2 = x^3 - \frac{12096}{529}x + \frac{24192}{529},$$
which satisfies the above property. $E$ attains a $K$-rational point of order $13$ over a sextic extension $K$ of $\mQ$. The first cubic extension $L/\mQ$ is given by adjoining a cubic root $\alpha$ of the $13$-division polynomial $\psi_{13}$, and then the second quadratic is given by adjoining the square root of the $f(\alpha)$. We also compute that $\cl (F) = 4$, $2$ splits in $\cO_K$, $\zeta_{13} \notin K$, and that $E/K$ is not supersingular modulo $\fl$ if $\ord_{\fl}(j_E) \geq 0$ where $\fl | 13$. Therefore, the elliptic curve $E$ and number field $K$ satisfy the conditions of Corollary \ref{D}.
\end{exam}

\begin{almexam}[$\ell = 13$]\label{6.3.4}
Suppose that elliptic curve $E/\mQ$ with mod 13 image conjugate to $B(13)$ will attain a $K$-rational point of order 13 over an extension of degree 12. The difficultly in verifying the conditions of Corollary \ref{D} is computing the class number and ramification indicies for the duodecic extension $K$.
\end{almexam}

Finally for $\ell = 37$, there is only one $E/\overline{\mQ}$ that we need to consider.
\begin{almexam}[$\ell = 37$]\label{6.3.7}
Suppose that $E/\mQ$ is the elliptic curve with $j$-invariant $-7\cdot 11^3$, which has affine equation
$$E\colon y^2 = x^3 - \frac{251559}{11045}x + \frac{503118}{11045}.$$
From \cite[Theorem~1.10.(ii)]{zywinapossible}, we know that the mod 37 image of $E$ is conjugate to 
$$
H := \begin{psmallmatrix} a^3 & * \\ 0 & *\end{psmallmatrix} \quad \mwhere a\in \mF_{37} .
$$
Since $\#(\mF_{37}\unit)^3 = 12$, $E$ attains a $K$-rational point of order 37 over a duodecic extension $K/\mQ$. 
As before, the difficultly in verifying the conditions of Corollary \ref{D} is computing the class number and ramification indicies for the duodecic extension $K$.
\end{almexam}

\bibliographystyle{amsalpha}
\def\bibfont{\small}
\bibliography{Morrow_SelmerTwist.bbl} 

\def\polhk#1{\setbox0=\hbox{#1}{\ooalign{\hidewidth
  \lower1.5ex\hbox{`}\hidewidth\crcr\unhbox0}}}
\providecommand{\bysame}{\leavevmode\hbox to3em{\hrulefill}\thinspace}
\providecommand{\MR}{\relax\ifhmode\unskip\space\fi MR }
\providecommand{\MRhref}[2]{%
  \href{http://www.ams.org/mathscinet-getitem?mr=#1}{#2}
}
\providecommand{\href}[2]{#2}
\begin{thebibliography}{Kam92}

\bibitem[BCP97]{MR1484478}
Wieb Bosma, John Cannon, and Catherine Playoust, \emph{The {M}agma algebra
  system. {I}. {T}he user language}, J. Symbolic Comput. \textbf{24} (1997),
  no.~3-4, 235--265, Computational algebra and number theory (London, 1993).
  \MR{MR1484478}

\bibitem[Der12]{derickx2012torsion}
Maarten Derickx, \emph{Torsion points on elliptic curves and gonalities of
  modular curves}, Master's thesis Universiteit Leiden (2012).

\bibitem[DKSS]{derickxtorsion}
Maarten Derickx, Sheldon Kamienny, William Stein, and Michael Stoll,
  \emph{Torsion points on elliptic curves over number fields of small degree},
  in preparation (private communication).

\bibitem[Epa]{stackexchange}
Epargyreus, \emph{Lack of ramification in cyclic extensions}, Mathematics Stack
  Exchange, \url{http://math.stackexchange.com/q/1585179}.

\bibitem[fM]{mathoverflow}
GH~from MO, \emph{Hecke characters and conductors}, MathOverflow,
  \url{http://mathoverflow.net/q/223873}.

\bibitem[Fre88]{frey1987selmer}
Georg Frey, \emph{On the {S}elmer group of twists of elliptic curves with
  $\mathbf{Q}$-rational torsion points}, vol.~XL, Canad. J. Math., 1988.

\bibitem[Kam92]{kamienny1992torsion}
Sheldon Kamienny, \emph{Torsion points on elliptic curves and $q$-coefficients
  of modular forms}, Inventiones mathematicae \textbf{109} (1992), no.~1,
  221--229.

\bibitem[Kat80]{katz1980galois}
Nicholas~M Katz, \emph{Galois properties of torsion points on abelian
  varieties}, Inventiones mathematicae \textbf{62} (1980), no.~3, 481--502.

\bibitem[KSS]{kamienny2011torsion}
Sheldon Kamienny, William Stein, and Michael Stoll, \emph{Torsion points on
  elliptic curves over quartic number fields}, preprint.

\bibitem[LR13]{alvarotorsion2013}
{\'A}lvaro Lozano-Robledo, \emph{On the field of definition of $p$-torsion
  points on elliptic curves over the rationals}, Mathematische Annalen
  \textbf{357} (2013), no.~1, 279--305.

\bibitem[Mai03]{mailhot2003Thesis}
James Mailhot, \emph{Selmer groups for elliptic curves with isogenies of prime
  degree}, Ph.D. thesis, University of Washington, 2003.

\bibitem[Maz77]{mazur1977rational}
Barry Mazur, \emph{Rational points on modular curves}, Modular functions of one
  variable V, Springer, 1977, pp.~107--148.

\bibitem[Mor]{morrowSelmer2015}
Jackson~S. Morrow, \emph{Electronic transcript of computations for the
  manuscript {``T}he {S}elmer group of twists of elliptic curves over {$K$}
  with {$K$}-rational torsion points"}, {A}vailable at
  \texttt{\url{https://drive.google.com/open?id=0Bx7T-L2ZBv-NNjllZFVtNXBsMUk}}.

\bibitem[N{\'e}r64]{NeronMinModel}
Andr{\'e} N{\'e}ron, \emph{Mod\`eles minimaux des vari\'et\'es ab\'eliennes sur
  les corps locaux et globaux}, Inst. Hautes \'Etudes Sci. Publ.Math. No.
  \textbf{21} (1964), 128.

\bibitem[Par03]{parent2003no}
Pierre Parent, \emph{No $17 $-torsion on elliptic curves over cubic number
  fields}, Journal de th{\'e}orie des nombres de Bordeaux \textbf{15} (2003),
  no.~3, 831--838.

\bibitem[Sil09]{silvermanAEC}
Joseph~H Silverman, \emph{The arithmetic of elliptic curves}, vol. 106,
  Springer, 2009.

\bibitem[Sut12]{sutherland2012torsion}
Andrew~V. Sutherland, \emph{Torsion subgroups of elliptic curves over number
  fields}, preprint (2012).

\bibitem[Was12]{washington2012introduction}
Lawrence~C Washington, \emph{Introduction to cyclotomic fields}, vol.~83,
  Springer Science \& Business Media, 2012.

\bibitem[Zyw15]{zywinapossible}
David Zywina, \emph{On the possible images of the mod $\ell$ representations
  associated to elliptic curves over $\mathbf{Q}$}, {A}vailable at
  \texttt{\url{http://www.math.cornell.edu/~zywina/papers/PossibleImages/index.html}}.

\end{thebibliography}
\end{document}